\documentclass[11pt]{amsart}

\title{Notes on Pointed Gromov-Hausdorff Convergence}

\author[D.~Jansen]{Dorothea Jansen}

\address{WWU M\"unster\\
  Mathematisches Institut\\
  Einsteinstr.~62\\
  D-48149 M\"unster\\ 
  Germany}

\email{d.jansen@uni-muenster.de}

\date{\today}



\thanks{These notes are based on the appendix of the author's PhD thesis \cite{jansenphd}.} 

\thanks{This work was supported by
  the Gottfried Wilhelm Leibniz-Preis of Prof.~Dr.~Burkhard Wilking 
  and the SFB 878: Groups, Geometry \& Actions, at the University of M\"unster.}

\usepackage[english]{babel}

\usepackage[utf8]{inputenc}

\usepackage[T1]{fontenc}

\usepackage{xcolor}
\definecolor{pantone3282}{RGB}{0,142,150} 
\definecolor{pantone369}{RGB}{122,181,22} 

\usepackage{nameref}
\usepackage{aliascnt}

\usepackage{xspace}

\usepackage{comment}

\usepackage[hypertexnames=true]{hyperref}
\hypersetup{
 citebordercolor = pantone369!70,
 linkbordercolor = pantone3282!70, 
 urlbordercolor  = pantone3282!70, 
 pdfauthor 	 = Dorothea Jansen
}

\usepackage{tikz}
\usetikzlibrary{arrows,decorations.pathmorphing,backgrounds,positioning,fit}

\newtheorem{thm}{Theorem}[section]
  
\newaliascnt{prop}{thm}
  \newtheorem{prop}[prop]{Proposition}
  
  \aliascntresetthe{prop}
\newaliascnt{lemma}{thm}
  \newtheorem{lemma}[lemma]{Lemma}
  
  \aliascntresetthe{lemma}
\newaliascnt{cor}{thm}
  \newtheorem{cor}[cor]{Corollary}
  
  \aliascntresetthe{cor}
\newcommand{\precptnessThm}{Gromov's Pre-compactness Theorem\xspace}

\theoremstyle{definition}
\newaliascnt{defn}{thm}
  \newtheorem{defn}[defn]{Definition}
  
  \aliascntresetthe{defn} 
\newaliascnt{exm}{thm}
  \newtheorem{exm}[exm]{Example}
   
  \aliascntresetthe{exm} 
\newtheorem*{rmk}{Remark}

\DeclareMathOperator{\nn}{\mathbb{N}}
\DeclareMathOperator{\qq}{\mathbb{Q}}
\DeclareMathOperator{\rr}{\mathbb{R}}

\DeclareMathOperator{\eps}{\varepsilon}
\renewcommand{\phi}{\varphi}

\DeclareMathOperator{\diam}{diam}
\DeclareMathOperator{\Ric}{Ric}

\newcommand{\B}{\bar{B}} 
\newcommand{\dH}{d_{\textit{H}}}
\newcommand{\dGH}{d_{\textit{GH}}}
\newcommand{\Isom}[1]{\textit{Appr}_{#1}}
 
\newcommand{\dghp}[5]{\dGH((\B_{#1}^{#2}(#3),#3),(\B_{#1}^{#4}(#5),#5))}

\newcommand{\Isomp}[6]{\textit{Appr}_{#1}((\B_{#2}^{#3}(#4),#4),(\B_{#2}^{#5}(#6),#6))}
\def\togh{\stackrel{\textit{\tiny GH}}\to}
\def\toghp{\stackrel{\textit{\tiny pGH}}\to}

\DeclareMathOperator{\pt}{pt} 
\DeclareMathOperator{\id}{id} 
\DeclareMathOperator{\im}{im} 
\def\limomega{\lim\nolimits_\omega}
\newcommand{\norm}[1]{\|#1\|}

\def\aand{\xspace\textrm{and}\xspace}
\def\as{\textrm{ as }}

\newcommand{\GH}{Gro\-mov-Haus\-dorff\xspace}

\newcommand{\myquote}[1]{\textquoteleft #1\textquoteright}

\begin{document}

 
\begin{abstract}
 The present article addresses to everyone
 who starts working with (pointed) \GH convergence.
 In the major part, 
 both \GH convergence of compact and of pointed metric spaces
 are introduced and investigated.
 Moreover, the relation of sublimits occurring with pointed \GH convergence
 and ultralimits is discussed. 
\end{abstract}


\maketitle
\tableofcontents


\GH distance is an often used tool for measuring 
how far two compact metric spaces are from being isometric.
This distance, which was introduced by Gromov in \cite{gromov-groups},
leads to the notion of \GH convergence which can be extended to non-compact metric spaces
and allows to draw conclusions about the properties of the spaces \myquote{near} to the limit space, 
if the limit space is well understood.
Many textbooks 
such as \cite[sections 7.3-7.5]{burago}, \cite[section 10.1]{petersen} 
and \cite[p.~70ff.]{bridson-haefliger}
give a (more or less) detailed introduction to the distance of compact metric spaces.
Some even more detailed proofs can be found in \cite{rong}.
Since the literature on convergence of non-compact metric spaces usually is less comprehensive,
this article treats the latter in detail. 
For the sake of completeness,
it also contains a detailed introduction to the compact case,
which is built on the literature cited above.

The first section deals with \GH distance of compact metric spaces. 
In addition, so called \GH approximations are introduced 
and the relation between those two terms is described. 
For both terms, a pointed and a non-pointed version is introduced, 
and it will be proven that these terms result in the same notion of convergence.

The second section deals with convergence of non-compact metric spaces, 
and consists of three parts: 
First, for compact length spaces it will be proven that this notion of convergence coincides 
with the one for compact spaces.
Secondly, several properties of pointed \GH convergence will be verified.
After that, a convergence notion for points will be introduced and studied.
Finally, convergence of (Lipschitz) maps will be investigated.

The third and final section deals with ultralimits, 
a more general tool to create \myquote{limit spaces}, 
and states some properties of those.
In particular, a strong correspondence between ultralimits 
and subsequences converging in the pointed \GH sense will be established.


\section{The compact case}\label{sec:GH-cpt}
Given a metric space, an interesting question is 
whether it is possible to assign each two subsets a distance 
such that this distance in turn defines a metric.
In \cite[Chapter VIII §6]{hausdorff}, Hausdorff answered this question 
by describing what nowadays is called the Hausdorff distance: 
For two subsets of a metric space, this is the minimal radius 
such that each subset is contained in the ball (with this radius) of the other subset.
This was extended by Gromov in \cite[section 6]{gromov-groups} to describe 
how far two compact metric spaces are from being isometric
by mapping two such spaces isometrically into a third one 
and measuring the Hausdorff distance of the images.
(In fact, one can restrict to embedding the two spaces isometrically into their disjoint union.)
This is the so called \GH distance.

\begin{defn}\label{def:dGH-cpt-npt}
 For bounded subsets $A$ and $B$ of a metric space $(X,d)$, 
 the \emph{Hausdorff distance of $A$ and $B$} is defined as
 \begin{align*}
 \dH^d(A,B) &:= \inf \{\eps > 0 \mid A \subseteq B^X_{\eps}(B)~\aand~B \subseteq B^X_{\eps}(A)\}
 \intertext{where $B^X_{\eps}(B) := \{x \in X \mid \exists b \in B: d(x,b) < \eps\}$. 
 For two compact metric spaces $(X,d_X)$ and $(Y,d_Y)$, 
 the \emph{\GH distance of $X$ and $Y$} is defined as}
 \dGH(X,Y) &:= \inf \{ \dH^d(X,Y) \mid d \text{ admissible metric on } X \amalg Y\},
 \end{align*}
 where a metric $d$ on the disjoint union $X \amalg Y$ is called \textit{admissible} 
 if it satisfies $d_{|X \times X} = d_X$ and $d_{|Y \times Y} = d_Y$.
\end{defn}

On the space of (non-empty) compact subspaces of $X$, this $\dH$ defines a metric, 
while $\dGH$ defines a metric on the set of isometry classes of (non-empty) compact metric spaces. 
This will be proven below. From now on, all metric spaces are assumed to be non-empty.
In order to compare two metric spaces with respect to some fixed base points, 
the pointed \GH distance is used. 

\begin{defn}\label{def:dGH-cpt-pt}
 Let $(X,d)$ be a metric space, $A, B \subseteq X$ bounded subsets 
 and $a \in A$, $b \in B$ base points. 
 The \emph{pointed Hausdorff distance of $(A,a)$ and $(B,b)$} is given by
 \begin{align*}
 \dH^d((A,a),(B,b)) &:= \dH^d(A,B) + d(a,b)
 \intertext{and the \emph{pointed \GH distance} between 
 two pointed compact metric spaces $(X,x_0)$ and $(Y,y_0)$ is defined as}
 \dGH((X,x_0),(Y,y_0)) &:= \inf \{ \dH^d((X,x_0),(Y,y_0)) \mid d \text{ adm.~on } X \amalg Y\}.
 \end{align*}
\end{defn}

As in the non-pointed case, the pointed \GH distance defines a metric 
on the set of isometry classes of (non-empty) pointed compact metric spaces. 
In order to prove this, a notion strongly related to the one of \GH distance is used.

\begin{defn}
 Let $(X,d_X)$ and $(Y,d_Y)$ be metric spaces and $\eps > 0$. 
 A pair of (not necessarily continuous) maps $f: X \to Y$ and $g : Y \to X$
 is called \emph{($\eps$-)\GH approximations} or \emph{$\eps$-approximations}
 if for all $x,x_1,x_2 \in X$ and $y,y_1,y_2 \in Y$,
 \begin{align*}
  &|d_X(x_1,x_2) - d_Y(f(x_1),f(x_2))| < \eps, &d_X(g \circ f (x), x) < \eps, \\
  &|d_Y(y_1,y_2) - d_X(g(y_1),g(y_2))| < \eps, &d_Y(f \circ g (y), y) < \eps.  
 \end{align*}
 The set of all such pairs is denoted by $\Isom{\eps}(X,Y)$. 
 In the pointed case, one restricts to pointed maps:
 For $p \in X$ and $q \in Y$,
 \begin{align*}
 \Isom{\eps}((X,p), (Y,q)) := \{(f,g) \in \Isom{\eps}(X,Y) \mid f(p) = q~\aand~g(q) = p\}.
 \end{align*}
\end{defn}

\begin{rmk}
 In the literature, \GH approximations often are not defined as pairs of maps 
 but as one map $f : X \to Y$ where $f$ has \emph{distortion less than $\eps$},
 i.e.~for all $x_1,x_2 \in X$, $f$ satisfies
 \[|d_Y(f(x_1),f(x_2)) - d_X(x_1,x_2)| < \eps, \]
 and $B_{\eps}(f(X)) = Y$. 
 Observe that $(f,g) \in \Isom{\eps}(X,Y)$ already implies that $f$ has these properties 
 (for the same $\eps$).
 
 In the following it will be seen that \GH distance less than $\eps$ 
 corresponds to $\eps$-approximations (up to a factor).
 The next proposition shows that (up to another factor) 
 the definition of \GH approximations used here can be replaced by the one described above.
\end{rmk}

\begin{prop}
 Let $f:(X,d_X) \to (Y,d_Y)$ be a map between metric spaces 
 with distortion smaller than $\eps > 0$. 
 Then there exists a map $g : f(X) \to X$ satisfying $(f,g) \in \Isom{\eps}(X,f(X))$. 
 Moreover, if $Y = B_{\eps}(f(X))$, 
 then there exists a map $h : Y \to X$ such that $(f,h) \in \Isom{3\eps}(X,Y)$.
\end{prop}

\begin{proof}
 For each $y \in f(X)$ choose some $g(y) \in f^{-1}(y)$.
 In particular, the such defined map $g$ satisfies $f \circ g = \id_{|f(X)}$.
 For $y_1,y_2 \in f(X)$,
 \begin{align*}
  &|d_X(g(y_1),g(y_2)) - d_Y(y_1,y_2)|
  \\&= |d_X(g(y_1),g(y_2)) - d_Y(f(g(y_1)),f(g(y_2)))|
  < \eps,
 \end{align*}
 and for $x \in X$,
 \begin{align*}
  d(x,g \circ f(x))
  = |d(x,g \circ f(x))) - d(f(x),f(g \circ f(x)))| < \eps.
 \end{align*}
 Thus, $(f,g) \in \Isom{\eps}(X,f(X))$.
 
 Now assume $Y = B_{\eps}(f(X))$. 
 For $y \in f(X)$, define $h(y) := g(y)$, otherwise, 
 choose $y' \in f(X)$ with $d_Y(y,y') < \eps$ and define $h(y) := y'$.
 By construction, $h \circ f = g \circ f$, 
 i.e.~for all $x \in X$, \[d_X(h \circ f (x),x) < \eps.\]
 For arbitrary $y \in Y$, using $f \circ g = \id_{|f(X)}$, 
 $f \circ h (y) = f \circ g (y') = y'$ 
 for $y' \in f(X) \cap B_{\eps}(y)$ as in the definition of $h$.
 Hence, \[d_Y(f \circ h(y),y) = d_Y(y',y) < \eps.\]
 Finally, for arbitrary $y_1,y_2 \in Y$,
 \begin{align*}
  &|d_X(h(y_1),h(y_2)) - d_Y(y_1,y_2)|
  \\&\leq |d_X(h(y_1),h(y_2)) - d_Y(f(h(y_1)),f(h(y_2)))| 
  \\&\quad + |d_Y(f(h(y_1)),f(h(y_2))) - d_Y(y_1,y_2)|
  \\&< \eps + d_Y(f \circ h(y_1),y_1) + d_Y(f \circ h(y_2),y_2)
  \\&< 3 \eps.\qedhere
 \end{align*}
\end{proof}

Next, a strong connection between existence of \GH approximations 
and the \GH distance will be proven.

\begin{prop}\label{dgh_small_iff_eps_approx}
 Let $X$ and $Y$ be compact metric spaces with base points $p \in X$ and $q \in Y$, respectively,
 and $\eps > 0$.
 \begin{enumerate}
  \item\label{dgh_small_iff_eps_approx--a} If $\dGH(X,Y) < \eps$, 
  then $\Isom{2\!\eps}(X,Y) \ne \emptyset$.
  \item\label{dgh_small_iff_eps_approx--b} If $\Isom{\eps}(X,Y) \ne \emptyset$, 
  then $\dGH(X,Y) \leq 2\eps$.
  \item\label{dgh_small_iff_eps_approx--c} If $\dGH((X,p),(Y,q)) < \eps$, 
  then $\Isom{2\!\eps}((X,p),(Y,q)) \ne \emptyset$.
  \item\label{dgh_small_iff_eps_approx--d} If $\Isom{\eps}((X,p),(Y,q)) \ne \emptyset$, 
  then $\dGH((X,p),(Y,q)) \leq 2\eps$.
 \end{enumerate}
\end{prop}

\begin{proof}
 As the proofs of \ref{dgh_small_iff_eps_approx--a} and \ref{dgh_small_iff_eps_approx--b},
 respectively, are very similar to, 
 but slightly easier than those 
 of \ref{dgh_small_iff_eps_approx--c} and \ref{dgh_small_iff_eps_approx--d}, respectively, 
 only the latter two are proven here.
 
 \par\smallskip\noindent\ref{dgh_small_iff_eps_approx--c}
 Let $0 < \delta < \eps - \dGH((X,p),(Y,q))$ 
 and choose an admissible metric $d$ with 
 \[\dH^d((X,p),(Y,q)) < \dGH((X,p),(Y,q)) + \delta < \eps.\]
 Then $d(p,q) < \eps$ on the one hand and $\dH^d(X,Y) < \eps$ on the other, 
 i.e.~for all $x \in X$ there exists $y_x \in Y$ that satisfies $d(x,y_x) < \eps$. 
 Analogously, for each $y \in Y$ there is $x_y \in X$ satisfying $d(y,x_y) < \eps$.
 Define $f:X \to Y$ and $g: Y \to X$ by
 \[ f(x) :=	\begin{cases} 
		  q & \text{if } x = p, \\ 
		  y_x & \text{otherwise,} 
		\end{cases}  
		\quad \quad \quad  
    g(y) := 	\begin{cases} 
		  p & \text{if } y = q, \\ 
		  x_y & \text{otherwise.} 
		\end{cases} \]
 As seen above, $d(f(x),x) < \eps$ for all $x \in X$. 
 Thus, for all $x,x' \in X$, 
 \begin{align*}
  |d_Y(f(x),f(x'))- d_X(x,x')|
  &\leq d(f(x),x)) + d(f(x'),x') 
  < 2\eps.
 \end{align*}
 Analogously, $|d_X(g(y),g(y'))- d_Y(y,y')| <2\eps$ for all $y,y' \in Y$.
 Similarly, for $x \in X$,
 \begin{align*}
  d_X(g \circ f (x), x) 
  &= d(g \circ f (x), x)\\
  & \leq d(g (f (x)), f(x)) + d(f(x),x)\\
  & < 2 \eps,
 \end{align*}
 as well as $d_Y(f \circ g(y),y) < 2 \eps$ for all $y \in Y$.
 Thus, \[(f,g) \in \Isom{2\!\eps}((X,p),(Y,q)).\]
 This proves \ref{dgh_small_iff_eps_approx--c}.

 \par\smallskip\noindent\ref{dgh_small_iff_eps_approx--d}
 Fix an arbitrary pair $(f,g) \in \Isom{\eps}((X,p),(Y,q))$. 
 The definition of an admissible metric $d: (X \amalg Y) \times (X \amalg Y) \to \rr$ 
 requires $d_{|X \times X} := d_X$, $d_{|Y \times Y} := d_Y$ 
 and $d(y,x) := d(x,y)$ for $x \in X$ and $y \in Y$.
 Hence, it suffices to define $d(x,y)$ for $x \in X$ and $y \in Y$. 
 Then $d$ is positive definite and symmetric by definition.
 Thus, in order to prove that $d$ is a metric, it remains to check the triangle inequality. 
 If done so, then $d$ is in fact an admissible metric.
 
 Define $d: (X \amalg Y) \times (X \amalg Y) \to \rr$ via 
 \[d(x,y) := \frac{\eps}{2} + \inf\{d_X(x,x') + d_Y(f(x'),y) \mid x' \in X\}\]
 for $x \in X$ and $y \in Y$.
 It remains to check the triangle inequality. 
 For $x_1,x_2 \in X$ and $y \in Y$,
 \begin{align*}
  &d(x_1,x_2)+ d(x_2,y)\\
  &= d_X(x_1,x_2) + \frac{\eps}{2} + \inf\{d_X(x_2,x') + d_Y(f(x'),y) \mid x' \in X \} \\
  &= \frac{\eps}{2} + \inf\{d_X(x_1,x_2) + d_X(x_2,x') + d_Y(f(x'),y) \mid x' \in X \} \\
  &\geq \frac{\eps}{2} + \inf\{d_X(x_1,x') + d_Y(f(x'),y) \mid x' \in X \} \\
  &=d(x_1,y)
  \intertext{and}
  &d(x_1,y)+ d(y,x_2)\\
  &= \eps + \inf\{d_X(x_1,x') +d_Y(f(x'),y) 
  \\& {\color{white} = \eps + \inf\{} 
   + d_X(x_2,x'') + d_Y(f(x''),y) \mid x',x'' \in X\} 
  \displaybreak[0]\\
  &\geq \eps + \inf\{d_X(x_1,x') + d_Y(f(x'),f(x'')) + d_X(x_2,x'') \mid x',x'' \in X\} 
  \displaybreak[0]\\
  &\geq \eps + \inf\{d_X(x_1,x') + (d_X(x',x'') - \eps) + d_X(x_2,x'') \mid x',x'' \in X\} 
  \displaybreak[0]\\
  &\geq \inf\{d_X(x_1,x_2) \mid x',x'' \in X\} \\
  &= d(x_1,x_2).
 \end{align*}
 The two remaining triangle inequalities
 $d(x,y_1) + d(y_1,y_2) \geq d(x,y_2)$ and $d(y_1,x) + d(x,y_2) \geq d(y_1,y_2)$,
 where $x \in X$ and $y_1,y_2 \in Y$, 
 can be proven analogously.

 Using this metric $d$,
 \begin{align*}
  &d(p,q)
  = \frac{\eps}{2} + \inf\{d_X(p,x') + d_Y(f(x'),q) \mid x' \in X\} = \frac{\eps}{2}
  \intertext{since 
  $0 \leq \inf\{d_X(p,x') + d_Y(f(x'),q) \mid x' \in X\} \leq d_X(p,p) + d_Y(f(p),q) = 0$.
  Furthermore, for $x \in X$,}
  &d(x,f(x)) 
  = \frac{\eps}{2} + \inf\{ d_X(x,x') + d_Y(f(x'),f(x)) \mid x' \in X \} = \frac{\eps}{2}
  \intertext{using $x' = x$. For $y \in Y$, this implies}
  &d(y,g(y)) 
  \leq d(y, f \circ g(y)) + d(f \circ g(y), g(y))
  < \eps + \frac{\eps}{2} = \frac{3 \eps}{2}.
 \end{align*}
 Thus, $X \subseteq B^d_{{\eps}/{2}}(f(X)) \subseteq B^d_{{3\eps}/{2}}(Y)$ 
 and $Y \subseteq B^d_{{3 \eps}/{2}}(X)$, 
 i.e.~$\dH^d(X,Y) \leq \frac{3 \eps}{2} $ and 
 \[\dGH((X,p),(Y,q)) \leq \dH^d((X,p),(Y,q)) = \dH^d(X,Y) + d(p,q) \leq 2 \eps.\]
 This proves \ref{dgh_small_iff_eps_approx--d}.
\end{proof}

Using these approximations, one can prove that the pointed \GH distance defines a metric.
Two pointed metric spaces $(X,p)$ and $(Y,q)$ are called \emph{isometric}
if there exists an isometry $f: X \to Y$ with $f(p) = q$.

\begin{prop}
On the space of isometry classes of (pointed) compact metric spaces, $\dGH$ defines a metric.
\end{prop}

\begin{proof}
In order to prove that the \GH distance indeed defines a metric, 
one needs that the Hausdorff distance defines a metric.
Therefore, this proof splits into several steps: 
First, the Hausdorff distance will be investigated. 
Then it will be proven that the \GH distance defines a pseudo-metric 
on the class of (pointed) compact metric spaces,
i.e.~it is not definite, but satisfies all the other properties of a metric.
Finally, it will be proven that this already defines a metric up to isometry.

\par\smallskip\noindent\textit{Step 1: $\dH$ defines a metric in the non-pointed case.}
Let $(X,d)$ be a metric space and $A,B,C \subseteq X$ be compact. 
First, prove that $\dH$ is a metric in the non-pointed case:

By definition, $\dH^d(B,A) = \dH^d(A,B)$, $\dH^d(A,B) \geq 0$ and $\dH^d(A,A) = 0$.
In order to prove the triangle inequality, 
define $r_1 := \dH^d(A,B) \geq 0$ and $r_2 := \dH^d(B,C) \geq 0$ 
and let $\eps > 0$ be arbitrary.
For $a \in A$ there exists $b \in B$ with $d(a,b) < r_1 + \eps$.
Furthermore, there is $c \in C$ with $d(b,c) < r_2 + \eps$.
Hence, $d(a,c) < r_1 + r_2 + 2 \eps$ and this proves $A \subseteq B_{r_1+r_2+2\eps}(C)$.
An analogous argumentation proves $C \subseteq B_{r_1+r_2+2\eps}(A)$. 
Therefore, $\dH^d(A,C) \leq r_1 + r_2 + 2 \eps$.
Since $\eps > 0$ was arbitrary, 
\[\dH^d(A,C) \leq r_1 + r_2 = \dH^d(A,B) + \dH^d(B,C).\]

Assume that $A \ne B$ and $\dH^d(A,B) = 0$. 
Without loss of generality, assume there exists $a \in A$ with $a \notin B$. 
In particular, $d(a,b) > 0$ for all $b \in B$.
Because $B$ is compact, 
$0 < \inf\{d(a,b) \mid b \in B\} \leq \dH^d(A,B)$,
and this is a contradiction.

\par\smallskip\noindent\textit{Step 2: $\dH$ defines a metric in the pointed case.}
Now fix $a \in A$, $b \in B$ and $c \in C$.
Since $\dH$ is a metric in the non-pointed case, 
\[\dH^d((A,a),(B,b)) = \dH^d(A,B) + d(a,b) \geq 0\] 
and equality holds if and only if $A=B$ and $a=b$.
Obviously, $\dH$ is symmetric and
  \begin{align*}
   \lefteqn{\dH^d((A,a),(B,b)) + \dH^d((B,b),(C,c))}\\
   &= \dH^d(A,B) + \dH^d(B,C) + d(a,b) + d(b,c) \\
   &\geq \dH^d(A,C) + d(a,c) \\
   &= \dH^d((A,a),(C,c)).
  \end{align*}
Thus, $\dH$ defines a metric.

\par\smallskip\noindent\textit{Step 3: $\dGH$ defines a pseudo-metric.}
 From now on, the proof restricts to the case of pointed metric spaces 
 since the other one can be done completely analogously.
 Obviously, $\dGH$ is non-negative and symmetric.
 It remains to prove the triangle inequality.
 Let $(X,x_0)$, $(Y,y_0)$ and $(Z,z_0)$ be pointed compact metric spaces.
 For arbitrary $\eps > 0$, 
 choose admissible metrics $d_{XY}$ on $X \amalg Y$ and $d_{YZ}$ on $Y \amalg Z$ such that
 \begin{align*}
  \dH^{d_{XY}}((X,x_0),(Y,y_0)) &< \dGH((X,x_0),(Y,y_0)) + \eps \quad\aand\\
  \dH^{d_{YZ}}((Y,y_0),(Z,z_0)) &< \dGH((Y,y_0),(Z,z_0)) + \eps.
 \intertext{Define an admissible metric $d_{XZ}$ on $X \amalg Z$ by}
  d_{XZ}(x,z) &= \inf\{d_{XY}(x,y) + d_{YZ}(y,z) \mid y \in Y \}.
 \end{align*}
 This actually defines a metric: 
 Since everything else is obvious,
 only the triangle inequality needs to be checked.
 If all regarded points are contained in $X$ or all in $Z$, there is nothing to prove.
 For $x_1,x_2 \in X$ and $z \in Z$,
 \begin{align*}
  &d_{XZ}(x_1,x_2) + d_{XZ}(x_2,z)\\
  &= d_X(x_1,x_2) + \inf\{d_{XY}(x_2,y') + d_{YZ}(y',z) \mid y' \in Y\} 
  \displaybreak[0]\\
  &= \inf\{d_{XY}(x_1,x_2) + d_{XY}(x_2,y') + d_{YZ}(y',z) \mid y' \in Y\}
  \displaybreak[0]\\
  &\geq \inf\{d_{XY}(x_1,y') + d_{YZ}(y',z) \mid y' \in Y\} \\
  &= d_{XZ}(x_1,z)
  \intertext{and} 
  &d_{XZ}(x_1,z) + d_{XZ}(z,x_2)\\
  &= \inf\{d_{XY}(x_1,y') + d_{YZ}(y',z) + d_{YZ}(z,y'') + d_{XY}(y'',x_2) \mid y',y'' \in Y\} 
  \displaybreak[0]\\
  &\geq \inf\{d_{XY}(x_1,y') + d_{Y}(y',y'') + d_{XY}(y'',x_2) \mid y',y'' \in Y\} 
  \displaybreak[0]\\
  &\geq \inf\{d_{XY}(x_1,y') + d_{XY}(y',x_2) \mid y' \in Y\}
  \displaybreak[0]\\
  &\geq d_X(x_1,x_2)\\
  &= d_{XZ}(x_1,x_2).
 \end{align*}
 The remaining triangle inequalities
 $d_{XZ}(z_1,z_2) + d_{XZ}(z_2,x) \geq d_{XZ}(z_1,x)$ 
 and $d(z_1,x) + d_{XZ}(x,z_2) \geq d_{XZ}(z_1,z_2)$,
 where $x \in X$ and $z_1,z_2 \in Z$, 
 can be proven analogously.
 
 With similar arguments, 
 one can prove that $d_{XYZ}$ defines an admissible metric on $X \amalg Y \amalg Z$ where
 \[
  d_{XYZ}(x,y) := 
  \begin{cases}
   d_{XY}(x,y) &\text{if } x,y \in X \amalg Y, \\
   d_{XZ}(x,y) &\text{if } x,y \in X \amalg Z, \\
   d_{YZ}(x,y) &\text{if } x,y \in Y \amalg Z.
  \end{cases}
 \]
 With those admissible metrics,
 \begin{align*}
  \lefteqn{\dGH((X,x_0),(Z,z_0))}\\
  & \leq \dH^{d_{XYZ}} (X,Z) + d_{XYZ} (x_0,z_0) \\
  &\leq \dH^{d_{XYZ}} (X,Y) + \dH^{d_{XYZ}} (Y,Z) + d_{XYZ} (x_0,y_0) + d_{XYZ} (y_0,z_0) \\
  &\leq \dH^{d_{XY}} (X,Y) + \dH^{d_{YZ}} (Y,Z) + d_{XY} (x_0,y_0) + d_{YZ} (y_0,z_0) \\
  &< \dGH((X,x_0),(Y,y_0)) + \dGH((Y,y_0),(Z,z_0)) + 2 \eps,
 \end{align*}
 where in the second last inequality the fact is used 
 that for every $r > 0$
 the inclusion $X \subseteq B^{d_{XY}}_r(Y)$ implies the inclusion $X \subseteq B^{d_{XYZ}}_r(Y)$.
 Now letting $\eps \to 0$ proves the triangle inequality for $\dGH$.
 
 \par\smallskip\noindent\textit{Step 4: $\dGH$ defines a metric up to isometry.}
 It is easy to see that the distance of isometric pointed compact spaces vanishes: 
 Let $(X,p)$ and $(Y,q)$ be isometric via isometries $f$ and $g$. 
 Then $(f,g) \in \Isom{\eps/2}((X,p),(Y,q))$
 for arbitrary $\eps > 0$. 
 By \autoref{dgh_small_iff_eps_approx}, $\dGH((X,p),(Y,q)) \leq \eps$. 
 Hence, $\dGH((X,p),(Y,q))=0$.
 
 Conversely, let $(X,p)$ and $(Y,q)$ be two pointed compact metric spaces 
 satisfying $\dGH((X,p),(Y,q)) = 0$.
 By definition, for each $n \geq 1$ 
 there is an admissible metric $d_n$ on $X \amalg Y$ with $\dH^{d_n}(X,Y) + d_n(p,q) < \frac{1}{n}$.
 Since $X$ is compact and thus separable, 
 there exists a countable dense subset $X' = \{x_i \mid i \in \nn\} \subseteq X$ with $x_0 = p$.
 
 Define $y^0_n := q$. 
 The constant sequence $(y^0_n)_{n \in \nn}$ converges to $q$, 
 and for each $n$, $d_n(x_0,y^0_n) = d_n(p,q) < \frac{1}{n}$.
 
 Because of $\dH^{d_n}(X,Y) < \frac{1}{n}$, 
 there exists some $y^1_n \in Y$ with $d_n(x_1, y^1_n) < \frac{1}{n}$. 
 Since $Y$ is compact, 
 $(y^1_n)_n$ has a convergent subsequence $(y^1_{n_{i}})_{i \in \nn}$ with some limit $y_1 \in Y$.
 Then 
 \[
  d_{n_{i}} (x_1,y_1) 
  \leq d_{n_{i}}(x_1, y^1_{n_{i}}) + d_{n_{i}}(y^1_{n_{i}}, y_1) 
  \to 0 \as i \to \infty.
 \]
 
 The same argument for $x_2$ gives a subsequence $d_{n_{i_j}}$ of $d_{n_{i}}$ 
 and a point $y_2 \in Y$ with $d_{n_{i_j}}(x_2,y_2) \to 0$ as $j \to \infty$. 
 By a diagonal argument, 
 there is a subsequence $d_l$ of $d_n$ and a sequence $(y_i)_{i \in \nn}$ with $y_0 = q$
 with $d_l(x_i,y_i) \to 0$ as $l \to \infty$ for all $i$.

 Define $f: X' \to Y$ by $f(x_i) := y_i$. Since the $d_l$ are admissible metrics, for each $l$,
 \begin{align*}
 d_Y(f(x_i), f(x_j))
 &= d_l(f(x_i), f(x_j))
 = d_l(y_i,y_j)
 \intertext{and}
 d_X(x_i,x_j)
 &= d_l(x_i,x_j).
 \end{align*}
 Therefore,
 \begin{align*}
 |d_Y(f(x_i), f(x_j)) - d_X(x_i,x_j)|
 &= |d_l(y_i,y_j) - d_l(x_i,x_j)| \\
 &\leq d_l(y_i,x_i) + d_l(x_j,y_j)) \\
 &\to 0 \as l \to \infty.
 \end{align*}
 Hence, $f$ is an isometry.
 Since $X'$ is dense, 
 $f$ can be extended uniquely to an isometric embedding $f : X \to Y$ with $f(p) = q$.
 With a similar construction and using a subsequence of $d_l$, 
 there is an isometric embedding $g : Y \to X$ with $g(q) = p$.
 After passing to this subsequence, for each $x$,
 \begin{align*}
 d_l(g \circ f (x), x)
 \leq d_l(g(f(x)),f(x)) + d_l(f(x), x)
 \to 0 \as {l \to \infty}.
 \end{align*}
 Thus, $f$ is an isometry with $f(p) = q$, i.e.~$(X,p)$ and $(Y,q)$ are isometric.
\qedhere
\end{proof}

The definitions of pointed and non-pointed Gromov-Haus\-dorff distance 
essentially give the same notion of convergence.
This will be proven next.

\begin{prop}\label{prop_GH:comparison_pointed_unpointed_compact}
 Let $X$ and $Y$ be compact metric spaces.
 \begin{enumerate}
  \item\label{prop_GH:comparison_pointed_unpointed_compact-a} 
  For each $x \in X$ and $y \in Y$, 
  \[\dGH(X,Y) \leq \dGH((X,x),(Y,y)).\] 
  \item\label{prop_GH:comparison_pointed_unpointed_compact-b}  
  For any $x \in X$ there exists $y \in Y$ such that 
  \[\dGH((X,x),(Y,y)) \leq 2 \cdot \dGH(X,Y).\]
 \end{enumerate}
\end{prop}

\begin{proof}
 Both statements follow easily from the definitions:
 \par\smallskip\noindent\ref{prop_GH:comparison_pointed_unpointed_compact-a}
 First, let $x \in X$ and $y \in Y$ be arbitrary. Then
  \begin{align*}
   \dGH(X,Y) 
    &= \inf \{ \dH^d(X,Y) \mid d \text{ admissible metric on } X \amalg Y\} \\
    &\leq \inf \{ \dH^d(X,Y) + d(x,y) \mid d \text{ admissible metric on } X \amalg Y\} 
    \displaybreak[0]\\
    &= \inf \{ \dH^d((X,x),(Y,y)) \mid d \text{ admissible metric on } X \amalg Y\} \\
    &= \dGH((X,x),(Y,y)).
  \end{align*}
  
 \par\smallskip\noindent\ref{prop_GH:comparison_pointed_unpointed_compact-b}
 Now let $r := \dGH(X,Y) \geq 0$.
 For arbitrary $n \in \nn$, let $d_n$ be an admissible metric on $X \amalg Y$ satisfying
 \[\dH^{d_n}(X,Y) < \dGH(X,Y) + \frac{1}{n} = r + \frac{1}{n}.\]
 Thus, $X \subseteq \B^{d_n}_{r+1/n}(Y)$, 
 i.e.~for given $x \in X$ there exists $y_n \in Y$ such that $d_n(x,y_n) \leq r+\frac{1}{n}$. 
 Since $Y$ is compact, 
 there exists a convergent subsequence $(y_{n_m})_{m \in \nn}$ of $(y_n)_{n \in \nn}$ 
 with limit $y \in Y$.
 Then
 \begin{align*}
  &\dH^{d_{n_m}}((X,x),(Y,y))\\
  &= \dH^{d_{n_m}}(X,Y) + d_{n_m}(x,y) \\
  &\leq r + \frac{1}{n_m} + d_{n_m}(x,y_{n_m}) + d_{n_m}(y_{n_m},y) \\
  &\leq 2r + \frac{2}{n_m} + d_Y(y_{n_m},y)
 \intertext{and}
  &\dGH((X,x),(Y,y)) \\
  &= \inf \{ \dH^d((X,x),(Y,y)) \mid d \text{ admissible metric on } X \amalg Y\} \\
  &\leq \inf \{ \dH^{d_{n_m}}((X,x),(Y,y)) \mid m \in \nn\} \\
  &\leq \inf \{ 2r + \frac{2}{n_m} + d_Y(y_{n_m},y) \mid m \in \nn\} \\
 & = 2r.
 \qedhere
 \end{align*}
\end{proof}

It is not hard to give examples where the inequality in
\autoref{prop_GH:comparison_pointed_unpointed_compact}
\ref{prop_GH:comparison_pointed_unpointed_compact-a}
is strict 
or where equality in \ref{prop_GH:comparison_pointed_unpointed_compact-b} 
holds for either all or none of the points.
In order to improve readability of the example,
the following two statements are proven first.
\begin{lemma}\label{lem:comparison_hausdorff_distance_pointed_unpointed}
 Let $(X,d)$ be a metric space, $a \in A \subseteq X$ and $b \in X$.
 Then $\dH((A,a), (\{b\},b)) \geq \dH((A,a), (\{a\},a)) = \sup\{d(a,a') \mid a' \in A\}$.
\end{lemma}
\begin{proof}
 First,
 recall
 \begin{align*}
  \dH(A,\{b\})
  &= \inf\{r > 0 \mid A \subseteq B_r(b), b \in B_r(A)\}
  \\& = \max\{ \inf\{r > 0 \mid A \subseteq B_r(b)\},
	       \inf\{r > 0 \mid b \in B_r(A)\} \}
  \\& = \max\{ \sup\{d(a',b) \mid a' \in A\}, d(A,b)\}.
 \intertext{In particular,
 \[\dH((A,a), (\{a\},a))
   = \dH(A,\{a\})
   = \sup\{d(a',a) \mid a' \in A\}.\]
 Moreover,}
  \dH((A,a), (\{b\},b))
  &= \dH(A,\{b\}) + d(a,b)
  \\& = \max\{ \sup\{d(a',b) \mid a' \in A\}, d(A,b)\}  + d(a,b) 
  \\&\geq \sup\{d(a',b) + d(a,b) \mid a' \in A\}
  \\&\geq \sup\{d(a',a) \mid a' \in A\}
  \\&= \dH((A,a), (\{a\},a)).
  \qedhere
 \end{align*}
\end{proof}

\begin{prop}\label{lem:comparison_compact_gromov_hausdorff_distance_pointed_unpointed}
 Let $(X,d_X)$ be a compact metric space and $x \in X$.
 Then 
 $\dGH((X,x),(\{\pt\},\pt)) = \sup\{d_X(x,p) \mid p \in X\}$.
\end{prop}
\begin{proof}
 By \autoref{lem:comparison_hausdorff_distance_pointed_unpointed},
 \begin{align*}
  \dGH((X,x),(\{\pt\},\pt))
  &= \inf \{\dH^d((X,x),(\{\pt\},\pt)) \mid d \text{ adm.~on } X \amalg \{\pt\}\}
  \\&\geq \inf \{\sup\{d(x,p) \mid p \in X\} \mid d \text{ adm.~on } X \amalg \{\pt\}\}
  \\&= \sup\{d_X(x,p) \mid p \in X\}
  \\&= \dH^{d_X}((X,x), (\{x\},x)).
  \intertext{On the other hand,}
  \dGH((X,x),(\{\pt\},\pt))
  &\leq \dH^{d_X}((X,x), (\{x\},x))
 \end{align*}
 and this proves the claim.
\end{proof}

The following examples proves 
that $\dGH((X,x),(Y,y))$ may attain any value between $\dGH(X,Y)$ and $2 \cdot \dGH(X,Y)$.
\begin{exm}
 Let $D^2 = \{x \in \rr^2 \mid \norm{x} \leq 1\}$ denote the disk of radius $1$ in $\rr^2$. 
 By \autoref{lem:comparison_compact_gromov_hausdorff_distance_pointed_unpointed},
 for arbitrary $x \in D^2$,
 \[ \dGH((D^2,x),(\{\pt\},\pt)) = \sup\{d_{D^2}(x,p) \mid p \in D^2\} = \norm{x} + 1.\]
 Hence, for any $\lambda \in [1,2]$, 
 every point $x$ with $\norm{x} = \lambda - 1$ satisfies
 \[\dGH((D^2,x), (\{\pt\},\pt)) = \lambda \cdot \dGH(D^2, \{\pt\}).\]
\end{exm}
In particular, two extreme cases occur 
in the situation of \autoref{prop_GH:comparison_pointed_unpointed_compact}:
For $X = D^2$, $Y = \{\pt\}$ and $x = (0,0) \in \rr^2$, 
there is no $y \in Y$ with 
\[\dGH((X,x),(Y,y)) = 2 \cdot \dGH(X,Y).\] 
On the contrary, in this case, 
$\dGH((X,x),(Y,y)) = \dGH(X,Y)$ for all $y \in Y$.
On the other hand, if $x = (1,0) \in \rr^2$, then
\[\dGH((X,x),(Y,y)) = 2 \cdot \dGH(X,Y)\] 
for all $y \in Y$.

\begin{defn}
 Let $(X,d_X,p)$ and $(X_i,d_{X_i},p_i)$, $i \in \nn$, be pointed compact metric spaces.
 \begin{enumerate}
  \item 
   If $\dGH(X_i, X) \to 0$ as $i \to \infty$, 
   then \emph{$X_i$ converges to $X$}.
  \item 
   If $\dGH((X_i,p_i), (X,p)) \to 0$ as $i \to \infty$, 
   then \emph{$(X_i,p_i)$ converges to $(X,p)$}.
 \end{enumerate}
 If $X_i$ converges to $X$, this is denoted by $X_i \to X$.
 If $(X_i,p_i)$ converges to $(X,p)$, this is denoted by $(X_i,p_i) \to (X,p)$.
\end{defn}

\begin{cor}\label{cor:compact_case:pointed=non-pointed_convergence}
 Let $(X,d_X)$ and $(X_i,d_{X_i})$, $i \in \nn$, be compact metric spaces.
 \begin{enumerate}
  \item 
   If $(X_i,x_i) \to (X,x)$ for some $x_i \in X_i$ and $x \in X$, 
   then $X_i \to X$ as well.
  \item 
   If $X_i \to X$ and $x \in X$, 
   then there exist points $x_i \in X_i$ such that $(X_i,x_i) \to (X,x)$.
 \end{enumerate}
\end{cor}
\begin{proof}
 This is a direct consequence of 
 \autoref{lem:comparison_compact_gromov_hausdorff_distance_pointed_unpointed}.
\end{proof}

Recall that a metric space $(X,d_X)$ is called \emph{length space} if
\[d(x,y) = \inf\{L(c) \mid c \text{ continuous curve from } x \text{ to } y\}\]
for any $x,y \in X$, where $L(c)$ denotes the length of $c$.

\begin{prop}\label{prop:cpt_length_spaces_cvg_to_length_spaces}
 A complete compact \GH limit of compact length spaces is a length space.
\end{prop}

In the proof, the following statement is used.

\begin{lemma}[{cf.~\cite[Theorem 2.4.16]{burago}}]\label{lem:charact_length_space}
 Let $(X,d)$ be a complete metric space. 
 Then $(X,d)$ is a length space
 if and only if for all $x,y \in X$ and $\eps > 0$ there exists an $\eps$-midpoint, 
 i.e.~a point $z \in X$ with
 $|2 d(x,z) - d(x,y)| \leq \eps$ and $|2 d(y,z) - d(x,y)| \leq \eps$,
\end{lemma}
\begin{proof}
 First, let $(X,d)$ be a length space and $x,y \in X$ and $\eps > 0$ be arbitrary. 
 Since $X$ is a length space, there exists a curve $c: [0,L] \to X$ with 
 $c(0) = x$, $c(L) = y$ and length $L(c) \leq d(x,y) + \eps$. 
 Without loss of generality, assume $c$ to be parametrised by arc length.
 In particular, $L = L(c) \leq d(x,y) + \eps$.

 Define $z := c(\frac{L}{2})$.
 Clearly,
 \begin{align*}
  2d(x,z) 
  \leq 2 \cdot L(c_{|[0,\frac{L}{2}]}) 
  = L \leq d(x,y) + \eps,
 \end{align*}
 and analogously, 
 $2d(y,z) \leq d(x,y) + \eps$. 
 Now assume $d(y,z) - d(x,z) > \eps$.
 Then 
 \begin{align*}
  d(x,y)
  \leq d(x,z) + d(z,y)
  < 2 d(y,z) - \eps
  \leq d(x,y),
 \end{align*}
 and this is a contradiction.
 Hence, 
 \begin{align*}
  d(x,y) - 2 d(x,z) 
  \leq d(y,z) - d(x,z)
  \leq \eps.
 \end{align*}
 Analogously, $|2 d(y,z) - d(x,y)| \leq \eps$. 
 
 Now let $X$ be a metric space 
 such that for all pairs of points and $\eps > 0$ there exists an $\eps$-midpoint,
 and let $x,y \in X$ be arbitrary.
 If for every $\eps>0$ there is a curve $\gamma$ connecting $x$ and $y$ 
 of length $L(\gamma) \leq d(x,y) + \eps$,
 then \[\inf\{L(\gamma) \mid \gamma~\text{connects}~x~\aand~y\} = d(x,y)\] and 
 this proves that $(X,d)$ is a length space.
 
 So, let $L := d(x,y)$, $\eps > 0$ be arbitrary and define $\gamma$ inductively as follows:
 First, let $\gamma(0) = x$ and $\gamma(1) = y$. 
 Now, assume $\gamma(\frac{k}{2^m})$ to be defined 
 for some $m \in \nn$ and all $k \in \nn$ with $0 \leq k \leq 2^m$.
 For odd $1 \leq k \leq 2^{m+1} - 1$, 
 let $\gamma(\frac{k}{2^{m+1}})$ be an
 $\frac{\eps}{2^{2m+1}}$-midpoint 
 of $\gamma(\frac{k-1}{2^{m+1}})$ and $\gamma(\frac{k+1}{2^{m+1}})$.
 
 Inductively,
 $d(\gamma(\frac{k}{2^m}), \gamma(\frac{k+1}{2^m})) 
 \leq \frac{L}{2^m} + \frac{\eps}{2^m} \cdot \sum_{i=1}^{m} \frac{1}{2^i}$: 
 For $m=0$, by definition, $d(\gamma(0), \gamma(1)) = L$.
 Let the statement be true for some $m \in \nn$, 
 and let $0 \leq k \leq 2^{m+1}-1$. 
 First assume $k = 2l + 1$ to be odd.
 Then 
 \begin{align*}
  2 \cdot d\big(\gamma\Big(\frac{k}{2^{m+1}}\Big), \gamma\Big(\frac{k+1}{2^{m+1}}\Big)\big) 
  & \leq \frac{\eps}{2^{2m+1}}
    + d\big(\gamma\Big(\frac{l}{2^{m}}\Big), \gamma\Big(\frac{l+1}{2^{m}}\Big)\big) 
  \\& \leq \frac{\eps}{2^{2m+1}}
    + \frac{L}{2^m} + \frac{\eps}{2^m} \cdot \sum_{i=1}^{m} \frac{1}{2^i}
  \\& = \frac{L}{2^m} + \frac{\eps}{2^m} \cdot \sum_{i=1}^{m+1} \frac{1}{2^i}.
 \end{align*}
 The proof for even $k$ can be done analogously.
 Observe
 \begin{align*}
  d\big(\gamma\Big(\frac{k}{2^m}\Big), \gamma\Big(\frac{k+1}{2^m}\Big)\big) 
  \leq \frac{L}{2^m} + \frac{\eps}{2^m} \cdot \sum_{i=1}^{m} \frac{1}{2^i}
  \leq \frac{L + \eps}{2^m}.
 \end{align*}
 Hence, for all $m \in \nn$ and $0 \leq k < l \leq 2^m$,
 \begin{align*}
  d\big(\gamma\Big(\frac{k}{2^m}\Big), \gamma\Big(\frac{l}{2^m}\Big)\big)
  &\leq \sum_{j=k}^{l-1} d\big(\gamma\Big(\frac{j}{2^m}\Big), \gamma\Big(\frac{j+1}{2^m}\Big)\big)
  \\&\leq \sum_{j=k}^{l-1} \frac{L+\eps}{2^m}
  = (L + \eps) \cdot \Big(\frac{l}{2^m} - \frac{k}{2^m}\Big).
 \end{align*}
 
 In particular, defined as a function on the dyadic numbers in $[0,L]$, $\gamma$ is Lipschitz. 
 Thus, it can be extended to a Lipschitz, hence continuous, curve $\gamma: [0,L] \to X$ 
 where $\gamma(t)$ is defined as the limit of $\gamma(t_n)$ for dyadic numbers $t_n \to t$. 
 For such $0 \leq s < t \leq L$ and dyadic numbers $s_n \to s$ and $t_n \to t$, 
 \begin{align*}
  d(\gamma(s), \gamma(t))
  &= \lim_{n \to \infty} d(\gamma(s_n), \gamma(t_n))
  \\&\leq \lim_{n \to \infty} (L + \eps) \cdot |t_n-s_n|
  = (L + \eps) \cdot (t-s).
 \end{align*}
 Therefore,
 \begin{align*}
  L(\gamma)
  &= \sup\Big\{ \sum_{i=0}^{N-1} d(\gamma(t_i), \gamma(t_{i+1})) 
   \mid N \in \nn, 0 = t_0 < t_1 < \ldots < t_N = 1 \Big\}
  \\&\leq \sup\Big\{ \sum_{i=0}^{N-1} (L + \eps) \cdot (t_{i+1}-t_i)
   \mid N \in \nn,0 = t_0 < t_1 < \ldots < t_N = 1 \Big\}
  \\&= L + \eps.
  \qedhere
 \end{align*}
\end{proof}

\begin{proof}[Proof of \autoref{prop:cpt_length_spaces_cvg_to_length_spaces}]
 Let $x,y \in X$ and $\eps > 0$ be arbitrary. 
 Applying \autoref{lem:charact_length_space},
 it is enough to find an $\eps$-midpoint $z$ of $x$ and $y$.
 
 Choose $i \in \nn$ such that $\dGH(X_i,X) < \frac{\eps}{12}$.
 By \autoref{dgh_small_iff_eps_approx}, 
 there exist $(f,g) \in \Isom{\eps/6}(X_i,X)$.
 Let $z'$ be an $\frac{\eps}{6}$-midpoint of $g(x)$ and $g(y)$, 
 and define $z := f(z')$.
 Then
 \begin{align*}
  |2 d_X(x,z) - d_X(x,y)|
  &\leq |2 d_X(x,z) - 2 d_{X_i}(g(x),g(z))|
  \\&\quad + |2 d_{X_i}(g(x),g(z)) - 2 d_{X_i}(g(x),z')|
  \\&\quad + |2 d_{X_i}(g(x),z') - d_{X_i}(g(x),g(y))|
  \\&\quad + |d_{X_i}(g(x),g(y)) - d_X(x,y)|
  \\& < 2 \cdot \frac{\eps}{6}
      + 2 \cdot d_{X_i}(g \circ f(z'),z')|
      + \frac{\eps}{6}
      + \frac{\eps}{6}
  \\& < \eps. 
 \end{align*}
 Analogously, $|2 d_X(y,z) - d_X(x,y)| < \eps$. 
\end{proof}

In general, the \GH distance of two subsets of the same metric space, 
equipped with the induced metric,
can be estimated by their Hausdorff distance.
If this metric space is a length space and the subsets are balls,
this estimate can be expressed by using the radii and the distance of the base points.
This uses the property of length spaces
that $r$-ball around a ball of radius $s$ 
coincides with the $r+s$ ball (around the same base point).

\begin{lemma}\label{lem_GH:ball_around_ball_is_ball_of_radii-sum}
 Let $(X,d)$ be a length space, $p \in X$ and $r,s > 0$.
 Then \[B_r(B_s(p)) = B_{r+s}(p).\]
\end{lemma}

\begin{proof}
 Let $q \in B_r(B_s(p))$, i.e.~there exists $x \in B_s(p)$ with $d(x,q) < r$.
 Then \[d(q,p) \leq d(q,x) + d(x,p) < r+s\] proves $B_r(B_s(p)) \subseteq B_{r+s}(p)$.
 In fact, this inclusion holds in every metric space.
 
 Conversely, let $q \in B_{r+s}(p)$. 
 Since $B_s(p) \subseteq B_r(B_s(p))$, 
 without loss of generality, 
 assume $q \in B_{r+s}(p) \setminus B_s(p)$.
 Let $l := d(p,q)$ denote the distance of $p$ and $q$.
 In particular, $s \leq l < r+s$.
 Fix a shortest geodesic $\gamma: [0,l] \to X$ with $\gamma(0) = p$ and $\gamma(l) = q$.
 Define $\eps := \frac{1}{2} \cdot \min\{s,r+s-l\} > 0$ and $t := s-\eps \in (0,s) \subseteq [0,l]$.
 Then 
 \begin{align*}
  d(\gamma(t),p) &= t < s 
  \intertext{and}
  d(\gamma(t),q) &= l-t = l-s + \eps < l-s+r+s-l = r. 
 \end{align*}
 Hence, $\gamma(t) \in B_s(p)$ and $q \in B_r(\gamma(t))$. 
 Thus, $B_{r+s}(p) \subseteq B_r(B_s(p))$.
\end{proof}

\begin{lemma}
 Let $(X,d)$ be a length space, $p,q \in X$, $r,s > 0$.
 Then \[\dH^d(\B_r(p), \B_s(q)) \leq d(p,q) + |r-s|.\]
\end{lemma}

\begin{proof}
 Let $\eps := d(p,q) + |r-s|$.
 If $\eps = 0$, the claim holds due to $p=q$ and $r=s$.
 Hence, assume $\eps > 0$. 
 Then, applying \autoref{lem_GH:ball_around_ball_is_ball_of_radii-sum},
 \[
  B_r(p) 
  \subseteq B_{d(p,q) + r}(q) 
  \subseteq B_{d(p,q) + |r-s|+s}(q) 
  = B_{\eps + s}(q) 
  = B_{\eps}(B_s(q)).
 \]
 Analogously, $B_s(q) \subseteq B_{\eps}(B_r(p))$.
 Therefore, \[\dH^d(\B_r(p),\B_s(q)) = \dH^d(B_r(p),B_s(q)) \leq \eps.\qedhere\]
\end{proof}

\begin{cor}\label{lem_GH:estimate_GH-distance_of_balls_in_same_space}
 Let $(X,d)$ be a length space, $p,q \in X$, $r,s > 0$. Then
 \begin{enumerate}
  \item $\dGH((\B^X_r(p),p),(\B^X_s(p),p)) \leq |r-s|$,
  \item $\dghp{r}{X}{p}{X}{q} \leq 2 d(p,q)$.
 \end{enumerate}
\end{cor}

The diameters of metric spaces with small \GH distance are almost the same. 
In particular, for a convergent sequence of metric spaces, 
their diameters converge to the diameter of the limit space.

\begin{prop}\label{prop:GH_small->diff_diam_small}
 For compact metric spaces $(X,d_X)$ and $(Y,d_Y)$, \[|\diam(X)-\diam(Y)| \leq 2 \dGH(X,Y).\]
 In particular, if $X_i \to X$ for compact metric spaces $(X_i,d_{X_i})$, $i \in \nn$, 
 then \[\diam(X_i) \to \diam(X).\]
\end{prop}

\begin{proof}
 Let $\eps := \dGH(X,Y)$, $\delta > 0$ and $d$ be an admissible metric on $X \amalg Y$ 
 such that \[\dH^d(X,Y) < \dGH(X,Y) + \delta = \eps + \delta.\] 
 This implies $Y \subseteq B^d_{\eps + \delta}(X)$. 
 Thus, for any $y_1, y_2 \in Y$ there are $x_1, x_2 \in X$ 
 with $d(x_i,y_i) < \eps + \delta$ for $1 \leq i \leq 2$.
 Hence, 
 \[
  d_Y(y_1,y_2) 
  \leq d(y_1,x_1) + d_X(x_1,x_2) + d(x_2,y_2) 
  < 2 \eps + 2 \delta + \diam(X).
 \]
 Therefore, 
 \[
  \diam(Y) 
  = \sup\{d_Y(y_1,y_2) \mid y_1,y_2 \in Y\} 
  \leq \diam(X) + 2 \eps + 2 \delta.
 \]
 Since $\delta > 0$ was arbitrary, $\diam(Y) \leq \diam(X) + 2 \eps$. 
 The other inequality can be proven analogously.
\end{proof}

\begin{cor}
 If $(X,d)$ is a compact metric space and $\{\pt\}$ the space consisting of only one point,
 then $\dGH(X,\{\pt\}) = \frac{1}{2} \cdot \diam(X)$.
\end{cor}

\begin{proof}
 By \autoref{prop:GH_small->diff_diam_small}, $\diam(X) \leq 2 \cdot \dGH(X,\{\pt\})$.
 Thus, only the other inequality has to be proven.
 
 Let $\delta := \frac{1}{2} \cdot \diam(X)$, 
 and define an admissible metric $d$ on the disjoint union $X \amalg \{\pt\}$ 
 by $d(x,\pt) := \delta$.
 As usually, only the triangle inequality needs to be checked. 
 For arbitrary $x_1,x_2 \in X$,
 \begin{align*}
  &d(x_1,x_2) + d(x_2,\pt) = d(x_1,x_2) + \delta \geq \delta = d(x_1,\pt) \quad \aand\\
  &d(x_1,\pt) + d(\pt,x_2) = 2 \delta = \diam(X) \geq d(x_1,x_2).
 \end{align*}
 Using this metric,
 \[\dGH(X,\{\pt\}) \leq \dH^d(X,\{\pt\}) = \delta.\qedhere\]
\end{proof}

For a metric space $(X,d_X)$, 
let $\lambda X$ denote the rescaled metric space $(\lambda X, d_{\lambda X}) := (X, \lambda d_X)$.
Rescaling of compact metric spaces behaves nicely under \GH distance.
For any $p \in X$ and $r > 0$, observe 
\[B_r^X(p) 
  = \{q \in X \mid d_X(q,p) < r\} 
  =  \{q \in X \mid \lambda d_X(q,p) < \lambda r\} 
  = B_{\lambda r}^{\lambda X}(p). 
\]

\begin{lemma}
Let $(X,d_X)$ and $(Y, d_Y)$ be compact metric spaces.
  
For the Hausdorff distance, $\dH^{\lambda X}= \lambda \cdot \dH^X$ 
(both in the standard and in the pointed case).
  
For the \GH distance, both $\dGH(\lambda X, \lambda Y) = \lambda \cdot  \dGH(X,Y)$ 
and, for all $x \in X$ and $y \in Y$, 
$\dGH((\lambda X,x), (\lambda Y,y)) = \lambda \cdot \dGH((X,x),(Y,y))$. 
\end{lemma}

\begin{proof}
  First, let $A,B \subseteq X$. Then
  \begin{align*}
   \dH^{\lambda X}(A,B)
   &= \inf\{\eps > 0 
       \mid A \subseteq B^{\lambda X}_{\eps}(B)~\aand~B \subseteq B^{\lambda X}_{\eps}(A)\} \\
   &= \inf\{\lambda \tilde{\eps} > 0 
       \mid A \subseteq B^X_{\tilde{\eps}}(B)~\aand~B \subseteq B^X_{\tilde{\eps}}(A)\} \\
   &= \lambda \cdot \inf\{\tilde{\eps} > 0 
       \mid A \subseteq B^X_{\tilde{\eps}}(B)~\aand~B \subseteq B^X_{\tilde{\eps}}(A)\} \\
   &= \lambda \cdot \dH^X(A,B).
  \intertext{Furthermore, for $a \in A$ and $b \in B$,}
   \dH^{\lambda X}((A,a),(B,b))
   &= \dH^{\lambda X}(A,B) + d_{\lambda X}(a,b) \\
   &= \lambda \cdot \dH^{X}(A,B) + \lambda \cdot d_{X}(a,b) \\
   &= \lambda \cdot \dH^{X}((A,a),(B,b)).
  \end{align*}
  
  By definition, an admissible metric $\tilde{d}$ on $\lambda X \amalg \lambda Y$ 
  is a metric on $X \amalg Y$ satisfying 
  $\tilde{d}_{|X \times X} = d_{\lambda X} = \lambda \cdot d_X$ 
  and $\tilde{d}_{|Y \times Y} = d_{\lambda Y} = \lambda \cdot d_Y$. 
  Furthermore, $d := \frac{1}{\lambda} \cdot \tilde{d}$ is a metric 
  if and only if $\tilde{d}$ is a metric. 
  In addition, this metric $d$ satisfies 
  $d_{|X \times X} =  \frac{1}{\lambda} \cdot \tilde{d}_{|X \times X} = d_X$ 
  and $d_{|Y \times Y} = d_Y$. 
  Thus, $d$ is an admissible metric on $X \amalg Y$. 
  On the other hand, using similar arguments, 
  if $d$ is an admissible metric on $X \amalg Y$, 
  then $\tilde{d} := \lambda \cdot d$ is an admissible metric on $\lambda X \amalg \lambda Y$.
  Hence,
  \begin{align*}
   \dGH(\lambda X, \lambda Y)
   &= \inf \{ \dH^{\tilde{d}}(\lambda X,\lambda Y) 
       \mid \tilde{d} \text{ admissible metric on } \lambda X \amalg \lambda Y\} \\
   &= \inf \{ \dH^{\lambda d}(\lambda X,\lambda Y) 
       \mid \lambda \cdot d \text{ admissible metric on } \lambda X \amalg \lambda Y\} \\
   &= \inf \{ \lambda \cdot \dH^{d}(\lambda X,\lambda Y) 
       \mid d \text{ admissible metric on } X \amalg Y\} \\
   &= \lambda \cdot \dGH(X,Y).
  \end{align*}
  Analogously, $\dGH((\lambda X,x), (\lambda Y,y)) = \lambda \cdot \dGH((X,x),(Y,y))$.
\end{proof}


\section{The non-compact case}\label{sec:GH-ncpt}

For non-compact metric spaces, the above way of defining a metric (up to isometry) does not work:
Using the Hausdorff distance as before on unbounded sets may give distance infinity. 
Thus, instead of defining a notion of distance for non-compact metric spaces,
convergence is defined by using compact subspaces of these spaces only. 
On these, the previous definitions can be applied.

A metric space is called \emph{proper} if all closed balls are compact. 
Throughout the remaining section, all metric spaces will assumed to be proper.
Notice that proper metric spaces are complete.

For a metric space $(X,d_X)$, $p \in X$ and $r > 0$,
let 
\[\B_r(p) := \{q \in X \mid d_X(p,q) \leq r\}\]
denote the closed ball of radius $r$ around $p$.

\begin{defn}\label{def:dGH-noncpt-pt}
 Let $(X,d_X,p)$ and $(X_i,d_{X_i},p_i)$, $i \in \nn$, be pointed proper metric spaces.
 If 
 \[\dghp{r}{X_i}{p_i}{X}{p} \to 0 \as {i \to \infty}\]
 for all $r > 0$,  
 where the balls are equipped with the restricted metric,
 then \emph{$(X_i,p_i)$ converges to $(X,p)$ (in the pointed \GH sense)}.
 If $(X_i,p_i)$ converges to $(X,p)$, 
 this is denoted by $(X_i,p_i) \to (X,p)$ 
 and $(X,p)$ is called the \emph{(pointed \GH) limit} of $(X_i,p_i)$.
 
 Frequently, a sequence $(X_i,p_i)$ does not converge itself but has a converging subsequence. 
 The limit of such a subsequence is called \emph{sublimit} of $(X_i,p_i)$, 
 and $(X_i,p_i)$ is said to \emph{subconverge} to this limit.
\end{defn}

Naturally, the question arises under which conditions a given sequence of metric spaces 
converges in the pointed \GH sense.
For mani\-folds, the following theorem by Gromov states 
that in some cases at least a (\GH) sublimit exists.
In \autoref{sec:ultralimits}, another, more general concept 
of creating and guaranteeing \myquote{limits} will be introduced.
It will turn out that these limits in fact are \GH sublimits as well.

\begin{thm}[\precptnessThm, {\cite[Cor.~1.11]{petersen}}]
 For $n \geq 2$, $\kappa \in \rr$ and $D > 0$, the following classes are pre-compact, 
 i.e.~every sequence in the class has a convergent subsequence 
 whose limit lies in the closure of this class:
\begin{enumerate}
 \item 
  The collection of $n$-dimensional closed Riemannian manifolds 
  with $\Ric \geq (n-1) \cdot \kappa$ and $\diam \leq D$.
 \item 
  The collection of $n$-dimensional pointed complete Riemannian manifolds 
  with $\Ric \geq (n-1) \cdot \kappa$. 
\end{enumerate}
\end{thm}

The section is structured as follows: 
In \autoref{sec:GH-ncpt--comparison_compact_case}, 
the compability of the definition of pointed \GH convergence in \autoref{def:dGH-noncpt-pt}
with the notion of convergence induced by the \GH distance of compact metric (length) spaces 
(\autoref{def:dGH-cpt-npt} and \autoref{def:dGH-cpt-pt}) is verified.
Subsequently, \autoref{sec:GH-ncpt--properties} deals with 
stating and verifying several properties of pointed \GH convergence.
In this context, convergence of points and convergence of maps, respectively,
are introduced in \autoref{sec:GH-ncpt--convergence_of_points} 
and \autoref{sec:GH-ncpt--convergence_of_maps}, respectively.


\subsection{Comparison with the compact case}\label{sec:GH-ncpt--comparison_compact_case}

Applied to compact length spaces, the convergence in the pointed \GH sense
coincides with the convergence of compact metric spaces 
in the pointed sense defined in the previous section.
Conversely, given (non-pointed) convergence as defined for compact metric spaces 
and a fixed base point in the limit space, 
there exist base points such that the spaces converge in the pointed \GH sense.

In order to prove this, one uses the fact that approximations can be restricted to smaller balls. 
This is shown in the following lemma. 
Another statement of the lemma is that base points can be changed in a certain way. 
This will be useful later on as well.

\begin{lemma}\label{lem:GHA_restrict_to_smaller_set_and_different_base_point}
Let $(X,d_X)$ and $(Y,d_Y)$ be length spaces, 
 \begin{enumerate}
 \item\label{lem:GHA_restrict_to_smaller_set_and_different_base_point--a}
 Let $p, p' \in X$, $q,q' \in Y$ and $R \geq r > 0$ satisfy 
 $\B^{X}_r(p') \subseteq \B^{X}_{R}(p)$ and $\B^Y_r(q') \subseteq \B^Y_{R}(q)$.
 Moreover, let $\eps > 0$,
 \[(f,g) \in \Isomp{\eps}{R}{X}{p}{Y}{q}\] 
 and $\delta := \max\{d(f(p'),q'), d(p',g(q'))\} \geq 0$.
 Then \[\Isomp{{4\eps+\delta}}{r}{X}{p'}{Y}{q'} \ne \emptyset\] 
 and \[\dghp{r}{X}{p'}{Y}{q'} \leq 8\eps+2\delta.\] 
 \item\label{lem:GHA_restrict_to_smaller_set_and_different_base_point--b} 
 For $p \in X$, $q \in Y$ and $R \geq r > 0$, 
 \[\dghp{r}{X}{p}{Y}{q} \leq 16\cdot\dghp{R}{X}{p}{Y}{q}.\]
\end{enumerate} 
\end{lemma}

\begin{proof}
 \par\smallskip\noindent\ref{lem:GHA_restrict_to_smaller_set_and_different_base_point--a}
 For simplicity, let $\delta_f := d(f(p'),q')$ and $\delta_g := d(p',g(q'))$. 
 In particular, $\delta = \max\{\delta_f,\delta_g\}$. Let $\tilde{\eps} := 4\eps + \delta$.
 As $\B^{X}_r(p') \subseteq \B^{X}_{R}(p)$, one can restrict $f$ to $\B^{X}_r(p')$. 
 For $x \in \B^{X}_r(p')$,
 \begin{align*}
  d_Y(f(x),q') 
  &\leq d_Y(f(x), f(p')) + d_Y(f(p'),q') \\
  &\leq (d_X(x,p') + \eps) + \delta_f \\
  &< r + \eps + \delta_f.
 \end{align*}
 Hence, $f(\B^X_r(p')) \subseteq \B^Y_{r + \eps + \delta_f}(q')$. 
 Analogously, one can prove the inclusion $g(\B^Y_r(q')) \subseteq \B^X_{r + \eps + \delta_g}(p')$. 
 Now modify $f$ and $g$ in order to obtain maps $\tilde{f}$ and $\tilde{g}$, respectively, 
 whose images are contained in $\B^Y_r(q')$ and $\B^X_r(p')$, respectively,
 such that $(\tilde{f},\tilde{g})$ are $\tilde{\eps}$-approximations:

 For $y \in \B^Y_{r+\eps+\delta_f}(q') \setminus \B^Y_r(q')$ 
 choose a shortest geodesic $c: [0,l] \to Y$ with $c(0)=q'$ and $c(1)=y$ 
 where $r < l := d_Y(y,q') \leq r+\eps+\delta_f$. 
 Then $d_Y(c(r),q') = r$,
 in particular, $c(r) \in \B^Y_r(q')$, 
 and for $\hat{y} := c(r)$, 
 \begin{align*}
  d(y,\hat{y}) 
  &= d_Y(y,q') - d_Y(\hat{y},q')\\
  &< (r + \eps + \delta_f) - r\\
  &= \eps + \delta_f.
 \end{align*}
 Using this, define $\tilde{f} : \B^X_{r}(p') \to \B^Y_{r}(q')$ by 
 \[\tilde{f}(x) := 
  \begin{cases}
   q'			&\text{if } x = p', \\
   f(x)			&\text{if } x \neq p'~\aand~f(x) \in \B^Y_r(q'), \\
   \widehat{f(x)}	&\text{if } x \neq p'~\aand~f(x) \notin \B^Y_r(q'). \\
  \end{cases}
 \]
 Since $d_Y(\tilde{f}(p'), f(p')) = d_Y(q', f(p')) = \delta_f < \eps + \delta_f$
 and by construction,
 \[d_Y(\tilde{f}(x),f(x)) < \eps + \delta_f\] for all $x \in \B_r^X(p')$.
 Similarly, define $\tilde{g} : \B^Y_r(q') \to \B^X_r(p')$.
 Using analogous arguments proves
 \[d_X(\tilde{g}(y), g(y)) < \eps + \delta_g\] for all $y \in \B^Y_r(q')$. 
 
 By definition, $\tilde{f}(p') = q'$ and $\tilde{g}(q') = p'$, 
 so it remains to prove that $(\tilde{f},\tilde{g})$ are $\tilde{\eps}$-approximations.
 By construction,
 \begin{align*} 
  &|d_X(x_1,x_2) - d_Y(\tilde{f}(x_1),\tilde{f}(x_2))|\\
  &\leq |d_X(x_1,x_2) - d_Y(f(x_1),f(x_2))| 
     + |d_Y(f(x_1),f(x_2)) - d_Y(\tilde{f}(x_1),\tilde{f}(x_2))|\\
  &< \eps + (d_Y(f(x_1),\tilde{f}(x_1)) + d_Y(f(x_2),\tilde{f}(x_2)))\\
  &< \eps + 2(\eps + \delta_f)\\
  &< \tilde{\eps},
 \intertext{where $x_1, x_2 \in \B^X_r(p')$. 
 Analogously, $|d_Y(y_1,y_2) - d_X(\tilde{g}(y_1),\tilde{g}(y_2))| < \tilde{\eps}$
 for arbitrary $y_1,y_2 \in \B^Y_r(q')$.
 Furthermore, for $x \in \B_r^X(p')$,}
  &d_X(x, \tilde{g} \circ \tilde{f}(x))\\
  &\leq d_X(x, g \circ f(x)) 
    + d_X(g \circ f(x), g \circ \tilde{f}(x)) 
    + d_X(g \circ \tilde{f}(x), \tilde{g} \circ \tilde{f}(x))\\
  &< \eps + (\eps + d_Y(f(x),\tilde{f}(x))) + (\eps + \delta_g)\\
  &< 4\eps + \delta_f + \delta_g\\
  &= \tilde{\eps}.
 \end{align*}
 Analogously, $d_Y(y, \tilde{f} \circ \tilde{g}(y)) < \tilde{\eps}$ for all $y \in \B_r^Y(q')$.
 Hence, 
 \[(\tilde{f},\tilde{g}) \in \Isomp{\tilde{\eps}}{r}{X}{p'}{Y}{q'},\] 
 and by \autoref{dgh_small_iff_eps_approx}, 
 \[\dghp{r}{X}{p'}{Y}{q'} \leq 2 \tilde{\eps}.\]

 \par\smallskip\noindent\ref{lem:GHA_restrict_to_smaller_set_and_different_base_point--b}
 Let $\delta > 0$ be arbitrary and $\eps := \dghp{R}{X}{p}{Y}{q} + \delta > 0$.
 By \autoref{dgh_small_iff_eps_approx}, 
 \[\Isomp{2\eps}{R}{X}{p}{Y}{q} \ne \emptyset,\]
 and by \ref{lem:GHA_restrict_to_smaller_set_and_different_base_point--a},
 \begin{align*}
  &\dghp{r}{X}{p}{Y}{q} 
  \\&\leq 16 \eps 
  \\&= 16\cdot\dghp{R}{X}{p}{Y}{q} + 16\,\delta.
 \end{align*}
 Since $\delta > 0$ was arbitrary, this implies the claim.
\end{proof}

In order to avoid confusion,
for the next two statements,
let $X_i \togh X$ and $(X_i,p_i) \togh (X,p)$, respectively,
denote the convergence of compact metric spaces 
in the sense of \autoref{def:dGH-cpt-npt} and \autoref{def:dGH-cpt-pt}, respectively.
Further, denote by $(X_i,p_i) \toghp (X,p)$ 
the convergence in the pointed \GH sense of \autoref{def:dGH-noncpt-pt}.

\begin{prop}\label{prop:pt-GH_convegence->diam_convergence}
 Let $(X,d_X,p)$ and $(X_i,d_{X_i},p_i)$, $i \in \nn$, be pointed compact length spaces
 with $(X_i,p_i) \toghp (X,p)$.
 Then $X_i \togh X$,
 in particular, $\diam(X_i) \to \diam(X)$.
\end{prop}

\begin{proof}
 Assume $(\diam(X_i))_{i \in \nn}$ is not bounded. 
 Let $r > \diam(X)$. 
 Without loss of generality, assume $\diam(X_i) > r$ for all $i \in \nn$.
 
 Let $0 < \eps < r-\diam(X)$
 and choose points $x_i, y_i \in B_r^{X_i}(p_i)$ 
 satisfying $d_{X_i}(x_i,y_i) \geq r - \frac{\eps}{2}$. 
 Let $\eps_i := 2\cdot\dGH((X_i,p_i),(X,p))$ 
 and fix approximations $(f_i,g_i) \in \Isom{\eps_i}((X_i,p_i),(X,p))$.
 Then 
 \[\diam(X) \geq d_{X}(f_i(x_i),f_i(y_i)) \geq r - \frac{\eps}{2} - \eps_i.\]
 Since this holds for all $i \in \nn$, 
 \[\diam(X) \geq r - \frac{\eps}{2} > \diam(X)  + \frac{\eps}{2}.\]
 This is a contradiction.
 Thus, there is an $R > \diam(X)$ with $\diam(X_i) < R$ for all $i \in \nn$. 
 Then 
 \begin{align*}
 \dGH(X_i,X) 
 &= \dGH(\B_R^{X_i}(p_i),\B_R^X(p)) \\
 &\leq \dghp{R}{X_i}{p_i}{X}{p} \\
 &\to 0 \as i \to \infty.
 \end{align*}
 Hence, $X_i \to X$. \autoref{prop:GH_small->diff_diam_small} implies the second part of the claim.
\end{proof}

\begin{cor}
 Let $(X,d_X,p)$ and $(X_i,d_{X_i},p_i)$, $i \in \nn$, be pointed compact length spaces.
 Then $(X_i,p_i) \togh (X,p)$ if and only if $(X_i,p_i) \toghp (X,p)$.
\end{cor}

\begin{proof}
The proof is done by proving both implications separately.
First, assume $(X_i,p_i) \togh (X,p)$ and let $r > 0$ be arbitrary.

By \autoref{prop:GH_small->diff_diam_small}, $\diam(X_i) \to \diam(X)$, 
i.e.~without loss of generality, assume a strict diameter bound $D$ on all spaces $X_i$ and $X$. 
In particular, for all $r \geq D$, 
$(\B^{X_i}_r(p_i),p_i) = (X_i,p_i)$ converges to $(X,p) = (\B^X_r(p),p)$. 

For $0 < r < D$, 
\begin{align*}
&\dghp{r}{X_i}{p_i}{X}{p} \\
&\leq 16 \cdot \dghp{D}{X_i}{p_i}{X}{p} \\
&= 16 \cdot \dGH((X_i,p_i),(X,p)) \\
&\to 0
\end{align*}
by \autoref{lem:GHA_restrict_to_smaller_set_and_different_base_point}. 
Hence, $(X_i,p_i) \toghp (X,p)$.

Now let $(X_i,p_i) \toghp (X,p)$.
By \autoref{prop:pt-GH_convegence->diam_convergence}, $\diam(X_i) \to \diam(X)$. 
Without loss of generality, assume $\diam (X_i) \leq 2 \diam(X) =: r$. 
Thus, \[\dGH((X_i,p_i), (X,p)) = \dghp{r}{X_i}{p_i}{X}{p} \to 0.\qedhere\]
\end{proof}

In particular, if $X_i, X$ are compact and $p \in X$, 
then, by \autoref{cor:compact_case:pointed=non-pointed_convergence}, 
there exist $p_i \in X_i$ such that $(X_i,p_i) \togh (X,p)$. 
Hence, $(X_i,p_i) \toghp (X,p)$.

From now on, let $(X_i,p_i) \to (X,p)$ denote convergence in the pointed \GH sense.


\subsection{Properties as in the compact case}\label{sec:GH-ncpt--properties}

This subsection deals with several properties which are familiar from the compact case.
First of all, 
the \GH distance defines a metric on the set of the isometry classes of compact metric spaces.
In the non-compact case,
the limit of pointed \GH convergence still is unique up to isometry. 

\begin{prop}\label{lem_GH:GH_limit_unique_up_to_pointed_isometry}
 Let $(X,d_X,p)$, $(Y,d_Y,q)$ and $(X_i,d_{X_i},p_i)$, $i \in \nn$, be pointed length spaces.
 Assume $(X_i,p_i) \to (X,p)$ and $(X_i,p_i) \to (Y,q)$.
 Then $(X,p)$ and $(Y,q)$ are isometric.
\end{prop}

\begin{proof}
 For every $r > 0$, both $\B_r^X(p)$ and $\B_r^Y(q)$ are limits of $\B_r^{X_i}(p_i)$, 
 and thus, there exists a (bijective) isometry $f_r: \B_r^X(p) \to \B_r^Y(q)$ with $f_r(p) = q$.
 Choose a countable dense subset $X' := \{x_0, x_1, x_2, \dots\}$ of $X$ with $x_0 = p$,
 fix $i \in \nn$
 and let $N_i$ be the minimal natural number with $d(x_i,q) < N_i$.
 Define $y_i^n := q$ if $n < N_i$ and $y_i^n := f_n(x_i)$ otherwise.
 For $n \geq N_i$, \[d_Y(y_i^n, q) = d_Y(f_n(x_i), f_n(p)) = d_X(x_i,p),\] 
 i.e.~$(y_i^n)_{n \in \nn}$ is a sequence in the compact subset $\B^Y_{d_X(x_i,p)}(q)$.
 By a diagonal argument, there exists a subsequence $(n_m)_{m \in \nn}$ of the natural numbers
 such that for every $i \in \nn$ 
 the sequence $(y_i^{n_m})_{m \in \nn}$ has a limit $y_i \in \B^Y_{d_X(x_i,p)}(q)$.
 In particular, $y_0^n = f_n(p) = q$ for all $n \in \nn$ implies $y_0 = q$.
 For $i,j \in \nn$, by construction,
 \begin{align*}
  d_Y(y_i,y_j)
    &= \lim_{m \to \infty} d_Y(y_i^{n_m}, y_j^{n_m})
  \\&= \lim_{m \to \infty} d_Y(f_{n_m}(x_i), f_{n_m}(x_j))
  \\&= d_X(x_i, x_j),
 \end{align*}
 i.e.~the map $\tilde{f}: X' \to Y$ defined by $\tilde{f}(x_i) := y_i$ 
 is an isometry with $\tilde{f}(p) = q$.

 As $Y$ is complete, 
 there exists an extension of $\tilde{f}$ to an isometry $f: X \to Y$ with $f(p) = q$:
 Let $x \in X$ be arbitrary.
 Since $X'$ was chosen to be dense, 
 there exists a sequence $(x_{i_j})_{j \in \nn}$ in $X'$ converging to $x$. 
 This is a Cauchy sequence, 
 hence, $(\tilde{f}(x_{i_j}))_{j \in \nn}$ is a Cauchy sequence as well 
 and has a limit $y =: f(x)$.
 
 This defines indeed an isometry $f: X \to Y$: 
 Let $x,x' \in X$ be arbitrary and $x_{i_j}$ and $x_{i_l}$, respectively, 
 be sequences in $X'$ converging to $x$ and $x'$, respectively.
 Then 
 \begin{align*}
 d_Y(f(x),f(x')) 
   &= \lim_{j,l \to \infty} d_Y(\tilde{f}(x_{i_j}),\tilde{f}(x_{i_l})) 
 \\&= \lim_{j,l \to \infty} d_X(x_{i_j},x_{i_l})
 \\&= d_X(x,x').
 \end{align*}
 Thus, $f$ is an isometry. It remains to prove that $f$ is bijective:

 Using a further subsequence $n_{m_a}$ and the inverse maps $f_{n_{m_a}}^{-1}$, 
 an isometry $g : Y \to X$ can be constructed analogously.
 For arbitrary $x \in X$,
 let $(y_{k_l})_{l \in \nn}$ be the sequence in the dense subset $Y' \subseteq Y$ 
 used in the construction of $g$ converging to $f(x) \in Y$.
 Then
 \begin{align*}
  d_X(g \circ f (x),x)
  &= \lim_{a \to \infty} \lim_{l,j \to \infty} d_X(f_{n_{m_a}}^{-1}(y_{k_l}),x_{i_j}) \\
  &= \lim_{a \to \infty} \lim_{l,j \to \infty} d_Y(y_{k_l},f_{n_{m_a}}(x_{i_j}))\\
  &= d_Y(f(x),f(x))
  = 0.
 \end{align*}
 Analogously, $f \circ g = \id$. 
 Thus, $f$ is bijective.
\end{proof}


As in the compact case, 
Gromov-Haus\-dorff convergence preserves being a length space.

\begin{prop}
 Let $(X_i,d_{X_i},p_i)$, $i \in \nn$, be pointed length spaces
 and $(X,d_X,p)$ be a pointed metric space.
 If $(X_i,p_i) \to (X,p)$, then $X$ is a length space.
\end{prop}

\begin{proof}
 Let $x,y \in X$ and $\eps > 0$ be arbitrary. 
 For $r := \max\{d_X(x,p), d_X(y,p)\}$
 choose $n \in \nn$ with $\dghp{r}{X_n}{p_n}{X}{p} < \frac{\eps}{12}$.
 The rest of the proof can be done completely analogously 
 to the one of \autoref{prop:cpt_length_spaces_cvg_to_length_spaces}.
\end{proof}


As in the compact case, 
in the non-compact case there is a correspondence between 
(pointed) \GH convergence and approximations.
In order to prove this, the following lemma is needed.

\begin{lemma}\label{prop:eps^(r_n)_n<=h(1/r_n)}
 For all $r > 0$, 
 let $(\eps^r_n)_{n \in \nn}$ be a monotonically decreasing null sequence, 
 and $h: \rr^{>0} \to \rr^{>0}$ a function with $\lim_{x \to 0} h(x) = 0$.
 Then there exists a sequence $(r_n)_{n \in \nn}$ 
 with $\lim_{n \to \infty} r_n = \infty$ 
 and $\eps^{r_n}_n \leq h\big(\frac{1}{r_n}\big)$ for almost all $n \in \nn$.
\end{lemma}

\begin{proof}
 Let $A := \{n \in \nn \mid \forall r > 0 : \eps^r_n > h\big(\frac{1}{r}\big)\}$
 denote the set of all natural numbers $n$ for which no such \myquote{$r_n$} can exist.
 This set is finite:
 Fix $r > 0$. 
 Then $\eps^r_n > h\big(\frac{1}{r}\big)$ for all $n \in A$, 
 but, since $(\eps^r_n)_{n \in \nn}$ is a null sequence, 
 this inequality only holds for finitely many $n$. 
 Hence, $A$ is finite.
 
 Without loss of generality, 
 assume that for each $n$ there is at least one $r > 0$ 
 such that $\eps^r_n \leq h\big(\frac{1}{r}\big)$.
 
 Let $R_n := \{r > 0 \mid \eps^r_n \leq h\big(\frac{1}{r}\big)\} \ne \emptyset$
 denote the set of all radii which are possible candidates for \myquote{$r_n$}. 
 Then $(R_n)_{n \in \nn}$ is an increasing sequence:
 Fix $r \in R_n$. 
 Since $(\eps^r_n)_{n \in \nn}$ is monotonically decreasing, 
 $\eps^r_{n+1} \leq \eps^r_n \leq h\big(\frac{1}{r}\big)$. 
 Thus, $r \in R_{n+1}$.
 
 Suppose that these sets are uniformly bounded, 
 i.e.~there exists $C > 0$ such that $\bigcup_{n \in \nn} R_n \subseteq [0,C]$. 
 Then 
 $\eps^r_n > h\big(\frac{1}{r}\big)$ 
 for all $n$ and all $r > C$.
 Consequently, for all $r > C$ 
 the sequence $(\eps^r_n)_{n \in \nn}$ is bounded below by $h\big(\frac{1}{r}\big)$. 
 This is a contradiction to $(\eps^r_n)_{n \in \nn}$ being a null sequence. 
 
 Therefore, $\bigcup_{n \in \nn} R_n$ is unbounded, 
 i.e.~for all $C > 0$ there exists some $N \in \nn$ 
 such that $R_j \not \subseteq [0,C]$ for all $j \geq N$.
 In particular, for all $k \in \nn$ there is a minimal $N_k \in \nn$ 
 such that for all $j \geq N_k$ there is some $r^k_j \in R_j$ with $r^k_j > k$.
 There are two cases:
 \begin{enumerate}
  \item[1.] Let $N_k \to \infty$. 
  For every $n \in \nn$, $n \geq N_0$, 
  there is some $k \in \nn$ with $N_k \leq n < N_{k+1}$. 
  Fix this $k$
  and define $r_n := r^k_n$ for some $r^k_n\in R_n$ satisfying $r^k_n> k$.
  Then, for arbitrary $k \in \nn$ and all $n \geq N_k$, $r_n> k$. 
  Thus, $r_n \to \infty$. 
  Furthermore, by choice, $\eps_n^{r_n} \leq h\big(\frac{1}{r_n}\big)$.
  \item[2.] Let $k_0 \in \nn$ such that $N_k = N_{k_0}$ for all $k \geq k_0$. 
  For $n < N_{k_0}$, define $r_n$ as in the first case.
  For $n = N+m \geq N_{k_0} = N_{k_0 + m}$, 
  choose any $r_n := r^{k_0+m}_n \in R_n \cap (k_0+m, \infty)$. 
  Then $r_n \to \infty$ and $\eps_n^{r_n} \leq h\big(\frac{1}{r_n}\big)$.
  \qedhere
 \end{enumerate}
\end{proof}

\begin{prop}\label{prop:dgh_small_iff_eps_approx_noncompact}
 Let $(X,d_X,p)$ and $(X_i,d_{X_i},p_i)$, $i \in \nn$, be length spaces.
 Then the following statements are equivalent.
 \begin{enumerate}
  \item\label{prop:dgh_small_iff_eps_approx_noncompact--a} 
  $(X_i,p_i) \to (X,p)$. 
  \item\label{prop:dgh_small_iff_eps_approx_noncompact--b} 
  For all functions $g: \rr^{>0} \to \rr^{>0}$ with $\lim_{x \to 0} g(x) = 0$ 
  there exists $r_i \to \infty$ with 
  \[\dghp{r_i}{X_i}{p_i}{X}{p} \leq g\Big(\frac{1}{r_i}\Big).\]
  \item\label{prop:dgh_small_iff_eps_approx_noncompact--c} 
  There exist $r_i \to \infty$ and $\eps_i \to 0$ with 
  \[\dghp{r_i}{X_i}{p_i}{X}{p} \leq \eps_i.\]
 \end{enumerate}
\end{prop}

\begin{proof} 
 The proof is done by verifying the implications 
 \ref{prop:dgh_small_iff_eps_approx_noncompact--a} 
 $\Rightarrow$ \ref{prop:dgh_small_iff_eps_approx_noncompact--b}, 
 \ref{prop:dgh_small_iff_eps_approx_noncompact--b} 
 $\Rightarrow$ \ref{prop:dgh_small_iff_eps_approx_noncompact--c} and 
 \ref{prop:dgh_small_iff_eps_approx_noncompact--c} 
 $\Rightarrow$ \ref{prop:dgh_small_iff_eps_approx_noncompact--a}.
 First, let $(X_i,p_i) \to (X,p)$ 
 and $g: \rr^{>0} \to \rr^{>0}$ with $\lim_{x \to 0} g(x) = 0$ be arbitrary.
 For fixed $r > 0$, define
 \begin{align*}
  \tilde{\eps}^r_i &:= \dghp{r}{X_i}{p_i}{X}{p} \to 0 \as i \to \infty
  \intertext{and}
  \eps^r_i &:= \sup\{ \tilde{\eps}^r_j \mid j \geq i\} \to 0 \as i \to \infty.
 \end{align*}  
 This sequence $(\eps^r_i)_{i \in \nn}$ is monotonically decreasing 
 and satisfies $\eps^r_i \geq \tilde{\eps}^r_i$.
 By \autoref{prop:eps^(r_n)_n<=h(1/r_n)}, 
 there exists $r_i \to \infty$ such that $\eps^{r_i}_i \leq g\big(\frac{1}{r_i}\big)$ 
 for all $i \in \nn$.
 In particular, 
 \[
  \dghp{r_i}{X_i}{p_i}{X}{p} 
  = \tilde{\eps}^{r_i}_i 
  \leq \eps^{r_i}_i 
  \leq g\Big(\frac{1}{r_i}\Big),
 \] 
 and this proves \ref{prop:dgh_small_iff_eps_approx_noncompact--b}.
 Obviously, \ref{prop:dgh_small_iff_eps_approx_noncompact--b} 
 implies \ref{prop:dgh_small_iff_eps_approx_noncompact--c} 
 via choosing $g:=\id$ and $\eps_i := \frac{1}{r_i}$.
 
 Finally, let $\dghp{r_i}{X_i}{p_i}{X}{p} \leq \eps_i$ for some $r_i \to \infty$ and $\eps_i \to 0$.
 Fix $r > 0$. 
 Let $i \in \nn$ be large enough such that $r < r_i$.
 Then
 \[\dghp{r}{X_i}{p_i}{X}{p} \leq 16 \eps_i,\]
 by \autoref{lem:GHA_restrict_to_smaller_set_and_different_base_point}, 
 and this implies the claim.
\end{proof}

\begin{cor}\label{cor:dgh_small_iff_eps_approx_noncompact}
 Let $(X,d_X,p)$ and $(X_i,d_{X_i},p_i)$, $i \in \nn$, be pointed length spaces.
 Then the following statements are equivalent. 
 \begin{enumerate}
  \item 
   $(X_i,p_i) \to (X,p)$.
  \item 
   There is $\eps_i \to 0$ such that 
   $\Isomp{\eps_i}{1/\eps_i}{X_i}{p_i}{X}{p} \ne \emptyset$ for all $i$.
  \item 
   There is $\eps_i \to 0$ such that 
   $\dghp{1/\eps_i}{X_i}{p_i}{X}{p} \leq \eps_i$ for all $i$.
 \end{enumerate}
\end{cor}

\begin{proof} 
 This is a direct implication 
 of \autoref{dgh_small_iff_eps_approx} 
 and \autoref{prop:dgh_small_iff_eps_approx_noncompact}.
\end{proof}


Similarly to the compact case, 
the \GH distance and convergence, respectively, is related to the diameters of the spaces:
On the one hand, 
the distance of balls in $X$ and $X \times Y$ are bounded from above by the diameter of $Y$.
Recall that in the special case of $X = \{\pt\}$, 
the (non-pointed) distance equals $\frac{1}{2} \diam(Y)$.
On the other hand, 
in the compact case it was proven that convergence of spaces implies convergence of the diameters. 
For length spaces, an analogous statement will be established.
\begin{prop}\label{prop-C}
  Let $(X,d_X,x_0)$ and $(Y,d_Y,y_0)$ be pointed metric spaces. 
  If $Y$ is compact,
  then \[\dghp{r}{X}{x_0}{X \times Y}{(x_0,y_0)} \leq \diam(Y)\] 
  for all $r > 0$.
\end{prop}

\begin{proof}
 It suffices to define an admissible metric 
 and to estimate the Hausdorff distance with respect to this metric.
   
 Let $\delta > 0$ be arbitrary. 
 Define an admissible metric $d$ on $(X \times Y) \amalg X$ by 
 \[ d((x,y),x') := \sqrt{d_X(x,x')^2 + d_Y(y,y_0)^2 + \delta^2}.\]
 As usual, the only tricky part is to prove the triangle inequality:
 By the Minkowski inequality, for $x_1,x_1',x_2,x_2' \in X$ and $y_1,y_2 \in Y$,
  \begin{align*}
   &d((x_1,y_1),x_1') + d(x_1',x_2')\\
   &=\sqrt{d_X(x_1,x_1')^2 + d_Y(y_1,y_0)^2 + \delta^2} + d_X(x_1',x_2') 
   \displaybreak[0]\\
   &\geq \sqrt{d_X(x_1,x_1')^2 + d_X(x_1',x_2')^2 + d_Y(y_1,y_0)^2 + \delta^2} 
   \displaybreak[0]\\
   &\geq \sqrt{d_X(x_1,x_2')^2 + d_Y(y_1,y_0)^2 + \delta^2} \\
   &= d((x_1,y_1),x_2').
  \end{align*}
  With completely analogous argumentation, one can prove the remaining inequalities
 \begin{align*}
   d(x_1',(x_1,y_1)) + d((x_1,y_1),x_2')	&\geq d(x_1',x_2')     ,
   \\
   d((x_1,y_1),(x_2,y_2)) + d((x_2,y_2),x_2')	&\geq d((x_1,y_1),x_2') 
   \quad\aand\displaybreak[0]\\
  d((x_1,y_1),x_1') + d(x_1',(x_2,y_2))	&\geq d((x_1,y_1),(x_2,y_2)).
  \end{align*}
  Now fix $r > 0$ and let $(x,y) \in \B_r^{X \times Y}((x_0,y_0))$ be arbitrary.
  In particular, $x \in \B_r^{X}(x_0)$. 
  Thus,
  \begin{align*}
  d((x,y), \B_r^{X}(x_0)) 
    &\leq d((x,y),x)  
  \\&= \sqrt{d_Y(y,y_0)^2 + \delta^2}
  \\&\leq \sqrt{\diam(Y)^2 + \delta^2}.
  \end{align*}
  Hence, 
  \[\B_r^{X \times Y}((x_0,y_0)) \subseteq \B^d_{\sqrt{\diam(Y)^2 + \delta^2}} (\B_r^{X}(x_0)).\]
  For arbitrary $x \in \B_r^{X}(x_0)$, 
  one has $d((x,y_0),(x_0,y_0)) = d_X(x,x_0) < r$, 
  and therefore, $(x,y_0) \in \B_r^{X \times Y}((x_0,y_0))$.
  Thus,
  \[d(x, \B_r^{X \times Y}(x_0,y_0)) \leq d(x,(x,y_0)) = \delta\]
  and \[\B_r^{X}(x_0) \subseteq \B^d_{\delta} (\B_r^{X \times Y}(x_0,y_0)).\]
  Hence,
  \begin{align*}
   &\dghp{r}{X}{x_0}{X \times Y}{(x_0,y_0)}\\
   &\leq \dH^d(\B_r^{X}(x_0),\B_r^{X \times Y}(x_0,y_0) ) \\
   &\leq \max\{\sqrt{\diam(Y)^2 + \delta^2}, \delta\}\\
   &= \sqrt{\diam(Y)^2 + \delta^2}.
  \end{align*}
 Since $\delta$ was arbitrary, this proves the claim.
\end{proof}

In order to prove the convergence of diameters, 
one needs the following property of length spaces of infinite diameter:
Any ball of radius $r$ has diameter at least $r$. 
Though it is easy to see this, for the sake of completeness, the proof is given first.
\begin{lemma}\label{lem_GH:r_ball_has_diam>=r}
 Let $(X,d,p)$ be a pointed length space and $0 < r < \frac{\diam(X)}{2}$.
 Then $\diam(\B_r^X(p)) \geq r$.
\end{lemma}

\begin{proof}
 Assume that $d(q,p) \leq r$ for all $q \in X$.
 Hence, $\B_r(p) = X$,
 and this implies $\diam(X) \leq 2r < \diam(X)$, which is a contradiction.

 Hence, there is $q_r \in X$ such that $l_r := d(q_r,p) > r$.
 Fix a minimising geodesic $\gamma: [0,l_r] \to X$ with $\gamma(0) = p$ and $\gamma(l_r) = q_r$.
 Then $d(p,\gamma(r)) = r$, hence, $\gamma(r) \in \B_r(p)$. 
 In particular, 
 $\diam(\B_r(p)) \geq d(p,\gamma(r)) = r$.
\end{proof}

\begin{prop}\label{prop:GH-small->diff_diam_small--noncompact}
 Let $(X,d_X,p)$ and $(X_i,d_{X_i},p_i)$, $i \in \nn$, be pointed length spaces.
 If $(X_i,p_i) \to (X,p)$,
 then $\diam(X_i) \to \diam(X)$. 
 (Here, both $\diam(X_i)$ tending to infinity 
 as well as the notion $\infty \to \infty$ are allowed.)
\end{prop}

\begin{proof}
 Let $\eps_i \to 0$ be as in \autoref{cor:dgh_small_iff_eps_approx_noncompact}
 with
 \[\dghp{1/\eps_i}{X_i}{p_i}{X}{p} \leq \eps_i.\]
 By \autoref{prop:GH_small->diff_diam_small},
 $|\diam(\B^{X_i}_{1/\eps_i}(p_i))  - \diam(\B^X_{1/\eps_i}(p))| \leq 2 \eps_i \to 0$.
 Distinguish the two cases of $X$ being bounded and unbounded, respectively.
 
 \textit{Case 1: $\diam(X) < \infty$.}
  Without loss of generality, assume $\diam(X) < \frac{1}{2\eps_i}$ for all $i \in \nn$.
  Then $X = B_{1/\eps_i}^X(p)$
  and 
  \[
   |\diam(\B^{X_i}_{1/\eps_i}(p_i))  - \diam(X)| 
   = |\diam(\B^{X_i}_{1/\eps_i}(p_i))  - \diam(\B^X_{1/\eps_i}(p))| 
   \to 0,
  \]
  in particular, $\diam(\B^{X_i}_{1/\eps_i}(p_i)) \to \diam(X)$ as $i \to \infty$.
  Without loss of generality,
  assume $\diam(\B^{X_i}_{1/\eps_i}(p_i)) \leq 2 \cdot \diam(X)$ for all $i \in \nn$.
  
  Let $r_i := \min\!\big\{ \frac{1}{\eps_i}, \frac{1}{3} \cdot \diam(X_i) \big\} 
  < \frac{1}{2} \cdot \diam(X_i)$.
  By \autoref{lem_GH:r_ball_has_diam>=r},
  \begin{align*}
  r_i 
  \leq \diam(B_{r_i}^{X_i}(p_i)) 
  \leq \diam(B_{1/\eps_i}^{X_i}(p_i)) 
  \leq 2 \cdot \diam(X)
  < \frac{1}{\eps_i}.
  \end{align*}
  Hence, $\diam(X_i) = 3 r_i \leq 6 \cdot \diam(X)$, the $X_i$ are compact
  and \autoref{prop:GH_small->diff_diam_small} implies the claim.

 \textit{Case 2: $\diam(X) = \infty$.}
  Assume there is a subsequence $(i_j)_{j \in \nn}$ and $C > 0$ 
  with $\diam(X_{i_j}) < C$ for all $j \in \nn$.
  Pass to this subsequence.
  After passing to a further subsequence,
  $C < \frac{1}{\eps_i}$ for all $i \in \nn$.
  Then $X_i = B_{1/\eps_i}^{X_i}(p_i)$ 
  and this implies $\diam(\B^{X_i}_{1/\eps_i}(p_i)) = \diam(X_i) < C$.
  Further, by \autoref{lem_GH:r_ball_has_diam>=r},
  $\diam(\B_{1/\eps_i}^X(p)) \geq \frac{1}{\eps_i}$
  and
  \begin{align*}
   |\diam(\B^{X_i}_{1/\eps_i}(p_i)) - \diam(\B^X_{1/\eps_i}(p))| 
   \geq \frac{1}{\eps_i} - C
   \to \infty.
  \end{align*}
  This is a contradiction. Hence, $\diam(X_i) \to \infty$.
\end{proof}


\GH convergence is compatible with rescaling:
Given a converging sequence of length spaces and a converging sequence of rescaling factors,
the rescaled sequence converges
and the limit space is the original one rescaled by the limit of the rescaling sequence.
More generally, given a converging sequence of metric spaces 
and some bounded sequence of rescaling factors, 
the sublimits of the rescaled sequence correspond exactly 
to the sublimits of the rescaling sequence.

For a metric space $(X,d)$, 
recall that $\alpha X$ denotes the rescaled metric space $(X,\alpha\,d)$.

\begin{prop}\label{prop-E}
 Let $(X,d_X,p)$ and $(X_i,d_{X_i},p_i)$, $i \in \nn$, be pointed length spaces
 and $r_i, r, \alpha_i, \alpha > 0$.
 \begin{enumerate} 
  \item\label{prop-E--a}
   If $(X_i,p_i) \to (X,p)$ and $r_i \to r$, 
   then $(\B_{r_i}^{X_i}(p_i),p_i) \to (\B_r^{X}(p),p)$.
  \item\label{prop-E--b} 
   If $\alpha_i \to \alpha$, 
   then $(\alpha_i X, p) \to (\alpha X, p)$.
  \item\label{prop-E--c}
   If $(X_i,p_i) \to (X,p)$ and $\alpha_i \to \alpha$, 
   then $(\alpha_i X_i,p_i) \to (\alpha X,p)$.
  \item\label{prop-E--d}
   If $(X_i,p_i) \to (X,p)$ and $(\alpha_i X_i,p_i) \to (Y,q)$, 
   then there is $\alpha$ such that $\alpha_i \to \alpha$ 
   and $(Y,q) \cong (\alpha X,p)$.
 \end{enumerate}
\end{prop}

\begin{proof}
 \ref{prop-E--a}
 By \autoref{lem_GH:estimate_GH-distance_of_balls_in_same_space},
 \[
  \dGH((\B_{r_i}^{X_i}(p_i),p_i), (\B_r^{X_i}(p_i),p_i))
  \leq |r - r_i| 
  \to 0,
 \]
 and the triangle inequality implies
 \begin{align*}
  &\dGH(\B_{r_i}^{X_i}(p_i), \B_r^{X}(p))
  \\&\leq \dGH(\B_{r_i}^{X_i}(p_i), \B_r^{X_i}(p_i)) + \dGH(\B_{r}^{X_i}(p_i), \B_r^{X}(p))
  \\&\to 0.
 \end{align*}
   
 \par\smallskip\noindent\ref{prop-E--b}
 Without loss of generality, let $\alpha = 1$. 
 First, let $X$ be compact.
 Define $f_i: X \to \alpha_i X$ and $g_i : \alpha_i X \to X$ 
 by $f_i(x) := x$ and $g_i(x) := x$ for all $x \in X$. 
 Furthermore, let $0 < \eps_i := 2 \cdot |\alpha_i - 1| \cdot \diam(X) \to 0$.
 For any $x,x' \in X$, 
 \begin{align*}
  |d_{\alpha_i X}(f_i(x),f_i(x')) - d_X(x,x')| &= |\alpha_i-1| \cdot d_X(x,x') < \eps_i.
  \intertext{Analogously,}
  |d_{\alpha_i X}(x,x') - d_X(g_i(x),g_i(x'))| & < \eps_i.
 \end{align*}
 Furthermore, $d_X(x, g_i \circ f_i(x)) = 0 < \eps_i$ 
 and $d_X(f_i \circ g_i(x),x) = 0 < \eps_i$.
 Thus, $(f_i,g_i) \in \Isom{\eps_i}((\alpha_i X,p), (X,p))$ 
 and $(\alpha_i X,p) \to (X,p)$.
 Now let $X$ be non-compact and $r > 0$. Then, using \ref{prop-E--a} and the compact case,
 \begin{align*}
  &\dGH((\B_r^{\alpha_i X}(p),p), (\B_r^X(p),p))\\
  &\leq \dGH((\B_r^{\alpha_i X}(p),p), (\B_{\alpha_i r}^{\alpha_i X}(p),p)) 
     + \dGH((\B_{\alpha_i r}^{\alpha_i X}(p),p), (\B_r^X(p),p)) \\
  &= \alpha_i \cdot \dGH((\B_{r/\alpha_i}^{X}(p),p), (\B_{r}^{X}(p),p)) 
     + \dGH((\alpha_i \B_{r}^{X}(p),p), (\B_r^X(p),p)) \\
  &\to 0.
 \end{align*}

 \par\smallskip\noindent\ref{prop-E--c}
 By the triangle inequality, for fixed $r > 0$,
 \begin{align*}
  &\dGH((\B_{r}^{\alpha_i X_i}(p_i),p_i), (\B_{r}^{\alpha X}(p),p)) \\
  &\leq\dGH((\B_{r}^{\alpha_i X_i}(p_i),p_i), (\B_{\alpha_i r/\alpha}^{\alpha_i X}(p),p)) \\&\quad
  +	\dGH((\B_{\alpha_i r/\alpha}^{\alpha_i X}(p),p), (\B_{r}^{\alpha_i X}(p),p)) \\&\quad
  +	\dGH((\B_{r}^{\alpha_i X}(p),p), (\B_{r}^{\alpha X}(p),p)). 
 \end{align*} 
 By \ref{prop-E--a},
 \begin{align*}
  &\dGH((\B_{r}^{\alpha_i X_i}(p_i),p_i), (\B_{\alpha_i r/\alpha}^{\alpha_i X}(p),p)) 
  \\&= \alpha_i \cdot \dGH((\B_{r/\alpha_i}^{X_i}(p_i),p_i), (\B_{r/\alpha}^{X}(p),p)) 
  \to 0,
 \end{align*}
 by \autoref{lem_GH:estimate_GH-distance_of_balls_in_same_space},
 \begin{align*} 
  \dGH((\B_{\alpha_i r/\alpha}^{\alpha_i X}(p),p), (\B_{r}^{\alpha_i X}(p),p)) 
  \leq |r - \frac{\alpha_i}{\alpha} \cdot r| 
  \to 0,
 \end{align*}
 and by \ref{prop-E--b}, 
 \[\dGH((\B_{r}^{\alpha_i X}(p),p), (\B_{r}^{\alpha X}(p),p)) \to 0.\] 
 Hence, 
 $(\B_{r}^{\alpha_i X_i}(p_i),p_i) \to (\B_{r}^{\alpha X}(p),p)$ for every $r > 0$.
 
 \par\smallskip\noindent\ref{prop-E--d}
 Let $\alpha$ be an arbitrary accumulation point of $(\alpha_i)_{i \in \nn}$. 
 Hence, for a subsequence $(i_j)_{j \in \nn}$,
 both $\alpha_{i_j} \to \alpha$, and by \ref{prop-E--c},
 $(\alpha_{i_j} X_{i_j}, p_{i_j}) \to (\alpha X, p) \as j \to \infty$.
 On the other hand, $(\alpha_{i_j} X_{i_j}, p_{i_j}) \to (Y, q) \as j \to \infty$.
 Thus, $(Y,q)$ and $(\alpha X, p)$ are isometric 
 (cf.~\autoref{lem_GH:GH_limit_unique_up_to_pointed_isometry}).
\end{proof}

\begin{cor}
 Let $(X,d_X,p)$ and $(X_i,d_{X_i},p_i)$, $i \in \nn$, be pointed length spaces
 and $(\alpha_i)_{i \in \nn}$ be a bounded sequence.
 If $(X_i,p_i) \to (X,p)$, 
 then the sublimits of $(\alpha_i X_i,p_i)$
 correspond to the $(\alpha X, p)$ 
 for exactly the accumulation points $\alpha$ of $(\alpha_i)_{i \in \nn}$.
\end{cor}

\begin{proof}
 Let $\alpha$ be an accumulation point of $(\alpha_i)_{i \in \nn}$ 
 and $(\alpha_{i_j})_{j \in \nn}$ be the subsequence converging to $\alpha$.
 Then $(X_{i_j},p_{i_j}) \to (X, p)$, and by \autoref{prop-E}, 
 \[(\alpha_{i_j} X_{i_j},p_{i_j}) \to (\alpha X, p).\]

 Now let $(Y,y)$ be a sublimit of $(\alpha_i X_i,p_i)$, 
 i.e.~$(\alpha_{i_j} X_{i_j},p_{i_j}) \to (Y,y)$ for some subsequence $(i_j)_{j \in \nn}$.
 Since $(\alpha_{i_j})_{j \in \nn}$ is a bounded sequence, 
 there exists a convergent subsequence $(\alpha_{i_{j_l}})_{l \in \nn}$ with limit $\alpha$.
 For this subsequence, $(\alpha_{i_{j_l}} X_{i_{j_l}},p_{i_{j_l}}) \to (Y,y)$, 
 and $(\alpha_{i_{j_l}} X_{i_{j_l}},p_{i_{j_l}}) \to (\alpha X, p)$ by the first part.
 Thus, $(Y,y)$ is isometric to $(\alpha X, p)$ 
 for an accumulation point $\alpha$ of $(\alpha_i)_{i \in \nn}$.
\end{proof}


\subsection{Convergence of points}\label{sec:GH-ncpt--convergence_of_points}

In the previous section, convergent sequences of pointed metric (length) spaces were studied.
Given such a sequence and using the corresponding approximations, 
a notion for convergence of points can be introduced. 

\begin{defn}\label{dfn:q_i->q}
 Let $(X,d_X,p)$ and $(X_i,d_{X_i},p_i)$, $i \in \nn$, be pointed length spaces.
 Assume $(X_i,p_i) \to (X,p)$ and let
 $\eps_i \to 0$ and 
 \[(f_i,g_i) \in \Isomp{\eps_i}{1/\eps_i}{X_i}{p_i}{X}{p}\]
 as in \autoref{cor:dgh_small_iff_eps_approx_noncompact}. 
 Let $q_i \in \B^{X_i}_{1/\eps_i}(p_i)$ and $q \in X$. 
 Then \emph{$q_i$ converges to $q$}, denoted by $q_i \to q$, 
 if $f_i(q_i)$ converges to $q$ (in $X$).
\end{defn}

For $(X_i, p_i) \to (X,p)$ as in the definition, $p_i \to p$ due to $f_i(p_i) = p$.
Moreover, for each $x \in X$ there exists such a sequence $x_i$ satisfying $x_i \to x$, 
e.g.~$x_i := g_i(x)$.


Convergence $q_i \to q$ depends on the choice of the underlying \GH approximations: 
Convergence with respect to one pair of approximations 
does not necessarily imply convergence for another, 
as the following example shows.

\begin{exm} 
 For $i \in \nn$, let $X_i = X = \mathbb{S}^2$ be the $2$-dimensional sphere, 
 $p_i=p=N$ the north pole and $q_i = q$ some fixed point on the equator.
 Let $\phi$ denote the rotation of $\mathbb{S}^2$ by $\frac{\pi}{2}$ fixing $p$
 and define $f_i = g_i = f_{2i}' = g_{2i}' = \id_{\mathbb{S}^2}$,
 $f_{2i+1}'=\phi$ and $g_{2i+1}' = \phi^{-1}$.
 
 Then both $(f_i,g_i)$ and $(f_i',g_i')$ are pointed isometries 
 between $(X_i,p_i)$ and $(X,p)$
 satisfying $f_i(q_i) = q$, but $f_{2i}'(q_{2i}) = q \neq \phi(q) = f_{2i+1}'(q_{2i+1})$.
 Hence, $f_i'(q_i)$ is not convergent at all, 
 but subconvergent with limits $q$ and $\phi(q)$.
\end{exm}

In this example, after replacing the approximations, two sublimits occur: 
One sublimit is the limit corresponding to the original approximations,
the other one is its image under an isometry of the limit space.
Since \GH convergence distinguishes spaces only up to isometry, 
concretely $(X,p) \cong (h(X),h(p)) = (X,h(p))$ for any isometry $h$, 
this can be interpreted as follows: 
If $q$ is a sublimit of $q_i$ with respect to one \GH approximation, 
then it is a sublimit for all \GH approximations.

This is a general fact as the subsequent lemma shows.
In order to prove this, the separability of a connected proper metric space is needed.
Though it is easy to see that such a space is separable, 
for completeness, the proof is given first.

\begin{lemma}\label{connected&proper=>separable}
 A connected proper metric space is separable.
\end{lemma}

\begin{proof}  
 Let $(X,p)$ be a connected proper metric space and let $p \in X$ be arbitrary.
 Then \[X = \bigcup_{q \in \qq \cap (0,\infty)} \B_q(p).\]
 As a compact set, every $\B_q(p)$ is separable where $q \in \qq$ is positive. 
 Therefore, there exists a countable dense subset $A_q \subseteq \B_q(p)$.
 Let $A :=  \bigcup_{q \in \qq \cap (0,\infty)} A_q$.
 This $A$ is countable, and for arbitrary $x \in X$ 
 there is a positive $q \in \qq$ such that $x \in \B_q(p)$,
 i.e.~there exists a sequence $x_n \in A_q \subseteq A$ converging to $x$.
 Thus, $A$ is dense in $X$, hence, $X$ is separable.
\end{proof}

\begin{lemma}
 Let $(X,d_X,p)$ and $(X_i,d_{X_i},p_i)$, $i \in \nn$, be pointed length spaces.
 Assume $(X_i, p_i) \to (X,p)$ and let $\eps_i,\eps_i' \to 0$, $r_i,r_i' \to \infty$ and
 \begin{align*}
  (f_i,g_i) &\in \Isomp{\eps_i}{r_i}{X_i}{p_i}{X}{p},\\ 
  (f_i',g_i') &\in \Isomp{\eps_i'}{r_i'}{X_i}{p_i}{X}{p}.
 \end{align*} 
 Let $q_i \in \B_{\min\{r_i,r_i'\}}^{X_i}(p_i)$ and $q \in X$.
 If $f_i(q_i) \to q$ and $q'$ is an accumulation point of $f_i'(q_i)$, 
 then there exists an isometry $h : X \to X$ such that $h(q)=q'$.
\end{lemma}

\begin{proof}
 Without loss of generality, let $r_i = r_i'$: 
 Otherwise, let $R_i := \min\{r_i,r_i'\}$ 
 and, by \autoref{lem:GHA_restrict_to_smaller_set_and_different_base_point} 
 and the construction in its proof,
 there are 
 \begin{align*}
  (\tilde{f}_i,\tilde{g}_i) &\in \Isomp{\eps_i}{R_i}{X_i}{p_i}{X}{p}\\
  (\tilde{f}_i',\tilde{g}_i') &\in \Isomp{\eps_i'}{R_i}{X_i}{p_i}{X}{p}
 \end{align*}
 with
 \begin{align*}
  \tilde{f}_i(q_i) \to q &\text{ if and only if } f_i(q_i) \to q,\\
  \tilde{f}'_i(q_i) \to q &\text{ if and only if } f_i'(q_i) \to q.
 \end{align*}
 Define $h_i, \bar{h}_i: \B_{r_i}^X(p) \to \B_{r_i}^X(p)$ by 
 \[h_i := f_i' \circ g_i \quad\aand\quad \bar{h}_i := f_i \circ g_i'.\]
 In particular, $h_i(p) = \bar{h}_i(p) = p$.
 For any $x,x' \in \B_{r_i}^X(p)$,
 \begin{align*}
  &|d_X(h_i(x),h_i(x')) - d_X(x,x')|\\
  &\leq |d_X(f_i'(g_i(x)),f_i'(g_i(x'))) - d_{X_i}(g_i(x),g_i(x'))| 
  \\&\quad+ |d_{X_i}(g_i(x),g_i(x')) - d_X(x,x')| \\
  &\leq \eps_i' + \eps_i \to 0.
  \intertext{Analogously, $|d_X(\bar{h}_i(x),\bar{h}_i(x')) - d_X(x,x')| \to 0$. Moreover,}
  &d_X(\bar{h}_i \circ h_i (x),x)\\
  &= d_X(f_i \circ g_i' \circ f_i' \circ g_i (x),x) \\
  &\leq d_{X_i}(g_i \circ f_i \circ g_i' \circ f_i' \circ g_i (x),g_i(x)) + \eps_i \\
  &\leq d_{X_i}(g_i' \circ f_i' \circ g_i (x),g_i(x)) + 2\eps_i \\
  &\leq d_{X_i}(g_i (x),g_i(x)) + 2\eps_i + \eps_i' \to 0,
 \end{align*}
 and analogously, $d_X(h_i \circ \bar{h}_i (x),x) \to 0$.
 Hence, if the $h_i$ and $\bar{h}_i$ (sub)converge (in some sense), 
 their corresponding (sub)limits are isometries fixing $p$ with $\bar{h} = h^{-1}$.
 
 The idea for proving subconvergence is to choose a countable dense subset $A \subseteq X$, 
 to define the sublimit of all $h_i(a)$ where $a \in A$ 
 and to extend this limit to a continuous map on $X$.
 Doing the same simultaneously for $\bar{h}_i$ 
 gives another sublimit that turns out to be the inverse of the first.
 In the end, identifying $X$ with itself using this isometry proves the claim.
 
 Choose a countable dense subset $A = \{a_n \mid n \in \nn\} \subseteq X$
 (cf.~\autoref{connected&proper=>separable}) 
 and, for $i$ large enough such that $d_X(a_n,p) \leq r_i$, 
 define $z_n^i := h_i(a_n)$ and $\bar{z}_n^i := \bar{h}_i(a_n)$. 
 Since \[d_X(z_n^i, p) = d_X(h_i(a_n),h_i(p)) \to d_X(a_n,p),\]
 the sequence $(d(z_n^i, p))_{i \in \nn}$ is bounded from above by some $R > 0$. 
 Hence, $z_n^i$ is contained in $\B_R^X(p)$, 
 and therefore, has a convergent subsequence. 
 An analogous argument proves subconvergence for $(\bar{z}_n^i)_{i \in \nn}$.
 Thus, using a diagonal argument, 
 there is a subsequence $(i_j)_{j \in \nn}$ such that 
 for any $n \in \nn$ the sequences $(z_n^{i_j})_{j \in \nn}$ and $(\bar{z}_n^{i_j})_{j \in \nn}$,
 respectively, converge to some $z_n \in X$ and $\bar{z}_n \in X$, respectively.
 
 Define $h(a_n) := z_n$ and $\bar{h}(a_n) := \bar{z}_n$. 
 In particular, 
 \[d_X(h(a_n),h(a_m)) = d_X(a_n,a_m) = d_X(\bar{h}(a_n),\bar{h}(a_m))\] 
 for all $n,m \in \nn$.
 For arbitrary $x \in X$, 
 choose a Cauchy sequence $(a_{n_k})_{k \in \nn}$ in $A$ converging to $x$
 and let 
 \[
  h(x) := \lim_{k \to \infty} h(a_{n_k}) 
  \quad\aand\quad 
  \bar{h}(x) := \lim_{k \to \infty} \bar{h}(a_{n_k}).
 \]
 In fact, for any $k \in \nn$,
 \begin{align*}
  &d_X(h_{i_j}(x),h(x))\\
  &\leq d_X(h_{i_j}(x), h_{i_j}(a_{n_k})) 
         + d_X(h_{i_j}(a_{n_k}),h(a_{n_k})) 
         + d_X(h(a_{n_k}), h(x))\\
  &\leq d_X(x, a_{n_k}) + \eps_{i_j} + \eps_{i_j}' 
         + d_X(h_{i_j}(a_{n_k}),h(a_{n_k})) 
         + d_X(h(a_{n_k}), h(x))\\
  &\to d_X(x, a_{n_k}) + d_X(h(a_{n_k}), h(x)) \as {j \to \infty}.
  \end{align*}
 Since this holds for every $k \in \nn$ 
 and $d_X(x, a_{n_k}) + d_X(h(a_{n_k}), h(x)) \to 0$ as $k \to \infty$, 
 \begin{align*}
  h_{i_j}(x) \to h(x) \as {j \to \infty}. 
 \end{align*}
 Analogously,
  $\bar{h}_{i_j}(x) \to \bar{h}(x) \as {j \to \infty}$.
 In particular, $\bar{h}_{i_j} \circ h_{i_j} \to \bar{h} \circ h$ and vice versa.
 Thus, $h$ is an isometry on $X$ with inverse $\bar{h}$.
 
 Now let $f_i(q_i) \to q$. Then 
 \begin{align*}
  d_X(f_{i_j}'(q_{i_j}),h(q)) 
  &\leq d_{X_{i_j}}(g_{i_j}' \circ f_{i_j}'(q_{i_j}),g_{i_j}' \circ h(q)) + \eps_{i_j}'\\
  &\leq d_{X_{i_j}}(q_{i_j},g_{i_j}' \circ h(q)) + 2 \eps_{i_j}'\\
  &\leq d_X(f_{i_j}(q_{i_j}),f_{i_j} \circ g_{i_j}' \circ h(q)) + 2 \eps_{i_j}' + \eps_{i_j}\\
  &\leq d_X(f_{i_j}(q_{i_j}),q) + d_X(q,\bar{h}_{i_j}' \circ h(q)) + 2 \eps_{i_j}' + \eps_{i_j}\\
  &\to 0 \as {j \to \infty}.
 \end{align*}
 This proves
 $f_{i_j}'(q_{i_j}) \to h(q) \as {j \to \infty}$.
\end{proof}


The following statements allow to change the base points of a given convergent sequence.

\begin{prop}\label{(X_i,p_i)->(X,p),q_i->q=>(X_i,q_i)->(X,q)}
 Let $(X,d_X,p)$ and $(X_i,d_{X_i},p_i)$, $i \in \nn$, be pointed length spaces,
 and let $q_i \in X_i$ and $q \in X$. 
 If $(X_i,p_i) \to (X,p)$ and $q_i \to q$, then $(X_i,q_i) \to (X,q)$.
\end{prop}

\begin{proof}
 The proof is an immediate consequence of 
 \autoref{lem:GHA_restrict_to_smaller_set_and_different_base_point} 
 and \autoref{prop:dgh_small_iff_eps_approx_noncompact}:
 Choose $\eps_i \to 0$ and $(f_i,g_i) \in \Isomp{\eps_i}{1/\eps_i}{X_i}{p_i}{X}{p}$ 
 as in \autoref{dfn:q_i->q} with $f_i(q_i) \to q$.
 In particular,
 \[
  d_{X_i}(q_i,g_i(q)) 
  \leq \eps_i + d_X(f_i(q_i), f_i(g_i(q))) 
  \leq 2 \eps_i + d_X(f_i(q_i), q) 
  \to 0.
 \]
 Hence, $\delta_i := \max\{d_X(f_i(q_i),q),d_{X_i}(q_i,g_i(q))\} \to 0$. 
 
 Since $f_i(p_i) = p$,
 \[ 
  d_{X_i}(p_i,q_i) 
  \leq \eps_i + d_X(p,q) + d_X(q,f_i(q_i)) 
  \to d_X(p,q).
 \] 
 Let $r > 0$ be arbitrary. 
 Fix $i$ large enough such that 
 $2(r + d_X(p,q)) \leq \frac{1}{\eps_i}$
 and such that $d_{X_i}(p_i,q_i) \leq 2 d_X(p,q)$ or $d_{X_i}(p_i,q_i) \leq r$, respectively,
 if $p \neq q$ or $p=q$, respectively. 
 In particular, 
 \begin{align*}
  &\B^{X_i}_r(q_i) 
   \subseteq \B^{X_i}_{r + d_{X_i}(p_i,q_i)}(p_i) 
   \subseteq \B^{X_i}_{1/\eps_i}(p_i), 
  \\
  &\B^X_r(q) 
   \subseteq \B^X_{r + d_X(p,q)}(p) 
   \subseteq \B^X_{1/\eps_i}(p)
 \end{align*}
 and $\Isomp{4\eps_i + \delta_i}{r}{X_i}{q_i}{X}{q} \ne \emptyset$ 
 by \autoref{lem:GHA_restrict_to_smaller_set_and_different_base_point}.
 By \autoref{dgh_small_iff_eps_approx}, 
 \[
  \dghp{r}{X_i}{q_i}{X}{q} 
  \leq 8\eps_i + 2 \delta_i 
  \to 0,
 \]
 and \autoref{prop:dgh_small_iff_eps_approx_noncompact} implies the claim.
\qedhere
\end{proof}

\begin{cor}
 Let $(X,d_X,p)$ and $(X_i,d_{X_i},p_i)$, $i \in \nn$, be pointed length spaces.
 Let $q_i \in X_i$ with $d_{X_i}(p_i,q_i) \to 0$.
 Assume $(X_i,p_i) \to (X,p)$. 
 Then $(X_i,q_i) \to (X,p)$.
\end{cor}

\begin{proof}
 Choose $\eps_i \to 0$ and $(f_i,g_i) \in \Isomp{\eps_i}{1/\eps_i}{X_i}{p_i}{X}{p}$ 
 as in \autoref{cor:dgh_small_iff_eps_approx_noncompact}.
 Then 
 \[
  d_X(f_i(q_i),p) 
  = d_X(f_i(q_i),f_i(p)) 
  \leq d_{X_i}(q_i,p_i) + \eps_i 
  \to 0.
 \] 
 Hence, $q_i \to p$, 
 and \autoref{(X_i,p_i)->(X,p),q_i->q=>(X_i,q_i)->(X,q)} implies the claim.
\end{proof}

\begin{cor}
 Let $(X,d_X,p)$ and $(X_i,d_{X_i},p_i)$, $i \in \nn$, be pointed length spaces.
 Let $q_i \in X_i$ with $d_{X_i}(p_i,q_i) \leq C$ for some $C>0$.
 If $(X_i,p_i) \to (X,p)$, 
 then there exists $q \in X$ such that $(X_i,q_i)$ subconverges to $(X,q)$.
\end{cor}

\begin{proof}
 Let $\eps_i \to 0$ and $(f_i,g_i) \in \Isomp{\eps_i}{1/\eps_i}{X_i}{p_i}{X}{p}$ 
 be as in \autoref{cor:dgh_small_iff_eps_approx_noncompact}.
 For $R > C$ there is $i_0 > 0$ such that $C + \eps_i \leq R$ for all $i \geq i_0$. 
 Therefore, $f_i(q_i) \in \B_R(p)$ for all $i \geq i_0$. 
 Since this ball is compact, 
 there exists a convergent subsequence with limit $q \in \B_R(p)$.
 After passing to this subsequence, $q_i \to q$, 
 and \autoref{(X_i,p_i)->(X,p),q_i->q=>(X_i,q_i)->(X,q)} implies the claim.
\end{proof}


\subsection{Convergence of maps}\label{sec:GH-ncpt--convergence_of_maps}

So far, statements about the convergence of metric spaces and points were made. 
But even statements about maps between those convergent space are possible: 
In fact, Lipschitz maps (sub)con\-verge (in some sense) to Lipschitz maps.
The proof of this seems to be rather technical, 
but in fact essentially only uses the same methods 
one can use to prove convergence of compact subsets (without bothering \precptnessThm).
Therefore, a proof of the latter is given in advance 
after establishing the following (technical) lemma.

\begin{lemma}\label{lem_GH:similar_isometries_imply_convergence}
 Let $(X,d_X,p)$ and $(X_i,d_{X_i},p_i)$, $i \in \nn$, be pointed length spaces. 
 Assume $(X_i,p_i) \to (X,p)$ and
 let $\eps_i \to 0$ 
 and \[(f_i,g_i) \in \Isomp{\eps_i}{1/\eps_i}{X_i}{p_i}{X}{p}\] 
 be as in \autoref{cor:dgh_small_iff_eps_approx_noncompact}.
 Moreover, let $A_i \subseteq B_{1/\eps_i}^{X_i}(p_i)$ and $A \subseteq X$ be compact 
 and $f_i': A_i \to A$, $g_i' : A \to A_i$ and $\delta_i \to 0$ satisfy
 \begin{align*}
  d_X(f_i'(x_i),f_i(x_i)) \leq \delta_i \quad\aand\quad
  d_{X_i}(g_i'(x),g_i(x)) \leq \delta_i
 \end{align*}
 for all $x_i \in A_i$ and $x \in A$.
 Then $A_i \to A$.
\end{lemma}

\begin{proof}
 Prove $(f_i',g_i') \in \Isom{2(\eps_i + \delta_i)}(A_i,A)$: 
 For $x_i^1, x_i^2 \in A_i$,
 \begin{align*}
  &|d_{X}(f_i'(x_i^1), f_i'(x_i^2)) - d_{X_i}(x_i^1, x_i^2)|
  \\& \leq |d_{X}(f_i'(x_i^1), f_i'(x_i^2)) - d_{X}(f_i(x_i^1), f_i(x_i^2))| 
  \\&\quad  + |d_{X}(f_i(x_i^1), f_i(x_i^2)) - d_{X_i}(x_i^1, x_i^2)|
  \\& < d_X(f_i'(x_i^1), f_i(x_i^1)) + d_X(f_i'(x_i^2),f_i(x_i^2)) + \eps_i
  \\&\leq \eps_i + 2 \delta_i.
 \end{align*}
 Analogously, $|d_{X_i}(g_i'(x^1), g_i'(x^2)) - d_{X}(x^1, x^2)| < \eps_i + 2 \delta_i$ 
 for all $x^1,x^2 \in A$.
 Moreover, for $x_i \in A_i$,
 \begin{align*}
  &d_{X_i}(g_i' \circ f_i' (x_i),x_i))
  \\& \leq d_{X_i}(g_i' \circ f_i' (x_i),g_i \circ f_i' (x_i)) 
  \\&\quad + d_{X_i}(g_i \circ f_i' (x_i),g_i \circ f_i(x_i)) + d_{X_i}(g_i \circ f_i(x_i),x_i)
  \\& < \delta_i + (d_{X_i}(f_i' (x_i),f_i(x_i)) + \eps_i) + \eps_i
  \\& \leq 2(\eps_i + \delta_i),
 \end{align*}
 and analogously, $d_{X}(f_i' \circ g_i'(x),x)) < 2(\eps_i + \delta_i)$ for all $x \in A$.
\end{proof}

\begin{prop}\label{lem_GH:compact_subets_converge}
 Let $(X,d_X,p)$ and $(X_i,d_{X_i},p_i)$, $i \in \nn$, be length spaces 
 such that $(X_i,p_i) \to (X,p)$
 and let $\eps_i \to 0$ and 
 \[(f_i,g_i) \in \Isomp{\eps_i}{1/\eps_i}{X_i}{p_i}{X}{p}\] 
 be as in \autoref{cor:dgh_small_iff_eps_approx_noncompact}.
 Let $K_i \in X_i$ be compact with $K_i \subseteq \B_R^{X_i}(p_i)$ for some $R > 0$.
 After passing to a subsequence, 
 there exists $K \subseteq \B_r(p)$ such that $K_i$ subconverges to $K$.
\end{prop}

\begin{proof}
 Without loss of generality, 
 assume $R \leq \frac{1}{\eps_i}$ and $\eps_i \leq 1$ for all $i \in \nn$.
 
 Let $x_i \in K_i \subseteq \B_R^{X_i}(p_i)$ be arbitrary. 
 Then $f_i(x_i) \in B_{R + \eps_i}^X(p) \subseteq \B_{R+1}^X(p)$.
 Hence, the sequence $(f_i(x_i))_{i \in \nn}$ is contained in a compact set, 
 and therefore has a convergent subsequent.
 Unfortunately, for different choices of $x_i$ different subsequences might converge.
 Therefore, a diagonal argument on countable dense subsets of the $K_i$ will be used.
 
 Let $A_i = \{a_i^n \mid n \in \nn\} \subseteq K_i$ be a countable dense subset. 
 As seen above, the sequence $(f_i(a_i^n))_{i \in \nn}$, where $n \in \nn$, 
 has a convergent subsequence with limit $y_n \in \B_{R+1}^X(p)$.
 Moreover, this subsequence can be chosen such that, 
 after passing to this subsequence, 
 $d_X(f_i(a_i^n),y_n) < \frac{\eps_i}{4}$.
 By a diagonal argument, 
 there exists a common subsequence 
 such that for every $n \in \nn$ there is $y_n \in \B_{R+1}(p)$ 
 with $d_X(f_i(a_i^n),y_n) < \frac{\eps_i}{4}$ for all $i \in \nn$.
 Pass to this subsequence.
 
 Define $A := \{y_n \mid n \in \nn\}$ as the set of all these limits 
 and let $K := \bar{A}$ denote its closure. 
 In particular, $K$ is compact.
 Define maps $f_i' : K_i \to K$ and $g_i' : K \to K_i$ in the following way:
 For $x_i \in A_i$, i.e.~$x_i = a_i^n$ for some $n \in \nn$, 
 define $f_i' (x_i):= y_n \in A \subseteq K$. 
 If $x_i \in K_i \setminus A_i$, 
 choose $a_i^n \in A_i$ with $d_{X_i}(x_i,a_i^n) < \frac{\eps_i}{4}$ 
 and define $f_i'(x_i) := y_n \in A \subseteq K$.
 In particular, 
 \begin{align*}
  d_X(f_i'(x_i),f_i(x_i)) 
  &\leq d_X(y_n,f_i(a_i^n)) + d_X(f_i(a_i^n),f_i(x_i)) 
  \\&<  \frac{\eps_i}{4} + (\eps_i + d_{X_i}(a_i^n,x_i)) 
  \\&< \frac{\eps_i}{4} + \Big(\eps_i + \frac{\eps_i}{4}\Big) 
  = \frac{3}{2} \eps_i.
 \end{align*}
 For $x \in A$, i.e.~$x = y_n$ for some $n \in \nn$, 
 define $g_i'(y_n) := a_i^n \in A_i \subseteq K_i$.
 For $x \in X \setminus A$, 
 choose $y_n \in A$ with $d_X(x,y_n) < \frac{\eps_i}{4}$ 
 and let $g_i'(x) := a_i^n \in A_i \subseteq K_i$.
 Then
 \begin{align*}
  d_{X_i}(g_i'(x),g_i(x)) = d_{X_i}(a_i^n,g_i(x)) 
  &< 2 \eps_i + d_X(f_i(a_i^n),x)
  \\&\leq 2 \eps_i + d_X(f_i(a_i^n),y_n) + d_X(y_n,x)
  \\&< \frac{5}{2}\eps_i.
 \end{align*}
 Now \autoref{lem_GH:similar_isometries_imply_convergence} implies the claim. 
\end{proof}

\begin{lemma}\label{lem:Lipschitz-maps_converge_to_Lipschitz-limit-map}
 Let $(X,d_X)$, $(Y,d_Y)$, $(X_i,d_{X_i})$ and $(Y_i,d_{Y_i})$, $i \in \nn$, 
 be compact length spaces such that $X_i \to X$ and $Y_i \to Y$.
 Moreover, let $\alpha > 0$, $K_i \subseteq X_i$ be compact subsets 
 and $f_i : K_i \to Y_i$ be $\alpha$-bi-Lipschitz.
 After passing to a subsequence, the following holds:
 \begin{enumerate}
  \item\label{lem:Lipschitz-maps_converge_to_Lipschitz-limit-map--a}
  There exist compact subsets $K \subseteq X$ and $K' \subseteq Y$ 
  which are \GH limits of $K_i$ and $f_i(K_i)$, respectively, 
  and an $\alpha$-bi-Lipschitz map $f : K \to K'$ with $f(K)=K'$. 
  \item\label{lem:Lipschitz-maps_converge_to_Lipschitz-limit-map--b}
  For any compact subset $L \subseteq K \subseteq X$ 
  there are compact subsets $L_i \subseteq K_i$ such that
  $L_i \to L$ and $f_i(L_i) \to f(L)$ in the \GH sense.
 \end{enumerate}
\end{lemma}

\newcommand{\figureone}{
  \begin{center}
   \begin{tikzpicture}[auto,>=stealth]
    \node (x2) 	
	  at (0,3.25)
	  {$X_i$};
    \node[rotate=90] (ss1)
	  at (0,2.6)	
	  {$\subseteq$};
    \node (s2)
	  at (0,1.75)
	  {$K_i$};
    \node (t2)
	  at (0,0)
	  {$f_i(K_i)$};
    \node[rotate=90] (ss2)
	  at (0,-0.75)
	  {$\supseteq$};
    \node (y2)
	  at (0,-1.5)
	  {$Y_i$};
    \node (x3)
	  at (2.5,3.25)
	  {$X$};
    \node[rotate=90] (ss3)
	  at (2.5,2.6)
	  {$\subseteq$};
    \node (s3)
	  at (2.5,1.75)
	  {$K$};
    \node (t3)
	  at (2.5,0)
	  {$K'$};
    \node[rotate=90] (ss4)
	  at (2.5,-0.75)
	  {$\supseteq$};
    \node (y3)
	  at (2.5,-1.5)
	  {$Y$};
    \node (empty)
	  at (-3.5,0)
	  {\quad};
    \path 
     (0.0,1.45)	edge[->] 			
	  node [left]  {$f_i$} 	
	  (0.0,0.3)	
     (2.5,1.45)	edge[->] 			
	  node [right] {$f$} 	
	  (2.5,0.3)
     (2.75,1.75)edge[->,bend left=80] 		
	  node [right] {$h_i= f_i^Y \circ f_i \circ g_i^X$} 
	  (2.75,0.0)
     (0.6,3.25)	edge[->] 			
	  node [left]  {} 	
	  (2.1,3.25)	
     (0.6,-1.5)	edge[->] 			
	  node [left]  {} 	
	  (2.1,-1.5)
     (0.6,1.825)edge[->,bend left  =10,dashed] 	
	  node [above] {$f_i^X$} 	
	  (2.1,1.825)
     (0.6,1.675)edge[<-,bend right =10] 	
	  node [below] {$g_i^X$} 	
	  (2.1,1.675)
     (0.6,0.075)edge[->,bend left  =10] 	
	  node [above] {$f_i^Y$} 	
	  (2.1,0.075)
     (0.6,-0.075)edge[<-,bend right =10,dashed] 
	  node [below] {$g_i^Y$} 	
	  (2.1,-0.075)
    ;
   \end{tikzpicture}
   \caption{Sets and maps used to construct $f: K \to K'$.}
   \label{pic:Lipschitz-maps_converge_to_Lipschitz-limit-map-a} 
  \end{center}
}
\newcommand{\figuretwo}{
  \begin{center}
   \begin{tikzpicture}[auto,>=stealth]
    \node (xi)
	  at (0,4.75)
	  {$X_i$};
    \node[rotate=90] (ss0)
	  at (0,4.0)
	  {$\subseteq$};
    \node (ki)
	  at (0,3.25)
	  {$K_i$};
    \node[rotate=90] (ss1)
	  at (0,2.5)
	  {$\subseteq$};
    \node (li)
	  at (-0.5,1.75)
	  {$L_i = \overline{g_i^X(L)}$};
    \node (fili)
	  at (0,0)
	  {$f_i(L_i)$};
    \node[rotate=90] (ss2)
	  at (0,-0.75)
	  {$\supseteq$};
    \node (fiki)
	  at (0,-1.5)
	  {$f_i(K_i)$};
    \node[rotate=90] (ss5)
	  at (0,-2.25)
	  {$\supseteq$};
    \node (yi)
	  at (0,-3)
	  {$Y_i$};
    \node (x)
	  at (2.5,4.75)
	  {$X$};
    \node[rotate=90] (ss0)
	  at (2.5,4.0)
	  {$\subseteq$};
    \node (k)
	  at (2.5,3.25)
	  {$K$};
    \node[rotate=90] (ss3)
	  at (2.5,2.5)
	  {$\subseteq$};
    \node (l)
	  at (2.5,1.75)
	  {$L$};
    \node (t3)
	  at (2.5,0)
	  {$f(L)$};
    \node[rotate=90] (ss4)
	  at (2.5,-0.75)
	  {$\supseteq$};
    \node (y3)
	  at (3.0,-1.5)
	  {$f(K) = K'$};
    \node[rotate=90] (ss6)
	  at (2.5,-2.25)
	  {$\supseteq$};
    \node (z3)
	  at (2.5,-3)
	  {$Y$};
    \path 
     (0.6,4.75) edge[->]
	  node [left]  {}
	  (2.1,4.75)	
     (0.6,3.325) edge[->,bend left  =05,dashed]
	  node [above] {$f_i^X$}
	  (2.1,3.325)
     (0.6,3.175) edge[<-,bend right =05]
	  node [below] {$g_i^X$}
	  (2.1,3.175)
     (0.6,1.825) edge[->,bend left  =05,dashed]
	  node [above] {$\tilde{f}_i^X$}
	  (2.1,1.825)
     (0.6,1.675) edge[<-,bend right =05]
	  node [below] {$\tilde{g}_i^X$}
	  (2.1,1.675)
     (0.0,1.45) edge[->]
	  node [left]  {${f_i}_{|L_i}$}
	  (0.0,0.3)	
     (2.5,1.45) edge[->]
	  node [right] {$f_L$}
	  (2.5,0.3)
     (0.6,0.075) edge[->,bend left  =05]
	  node [above] {$\tilde{f}_i^Y$}
	  (2.1,0.075)
     (0.6,-0.075) edge[<-,bend right =05,dashed]
	  node [below] {$\tilde{g}_i^Y$}
	  (2.1,-0.075)
     (0.6,-1.425) edge[->,bend left  =05]
	  node [above] {$f_i^Y$}
	  (2.1,-1.425)
     (0.6,-1.575) edge[<-,bend right =05,dashed]
	  node [below] {$g_i^Y$}
	  (2.1,-1.575)
     (0.6,-3) edge[->]
	  node [left]  {}
	  (2.1,-3)
    ;
   \end{tikzpicture}
   \caption{Sets and maps used to construct $L_i \to L$.}
   \label{pic:Lipschitz-maps_converge_to_Lipschitz-limit-map-b} 
  \end{center}
}

\begin{proof}
 \par\smallskip\noindent\ref{lem:Lipschitz-maps_converge_to_Lipschitz-limit-map--a}
 In order to prove the first part,
 pass to the subsequence of \autoref{lem_GH:compact_subets_converge}.
 Then
 there are compact sets $K \subseteq X$ and $K' \subseteq Y$ 
 such that $K_i \to K$ and $f_i(K_i) \to K'$.
 For these, fix $\eps_i \to 0$,
 $(f_i^X, g_i^X) \in \Isom{\eps_i}(K_i,K)$ 
 and $(f_i^Y, g_i^Y) \in \Isom{\eps_i}(f_i(K_i),K')$,
 cf.~\autoref{pic:Lipschitz-maps_converge_to_Lipschitz-limit-map-a}.
 
 \begin{figure}[t]
  \figureone
 \end{figure}

 The idea is to define $f$ as a limit of $h_i := f_i^Y \circ f_i \circ g_i^X: K \to K'$: 
 For $x,x' \in K$,
 \begin{align*}
  d_Y(h_i(x),h_i(x'))
  &= d_Y(f_i^Y \circ f_i \circ g_i^X (x), f_i^Y \circ f_i \circ g_i^X(x')) \\
  &\leq \eps_i + d_{Y_i}(f_i \circ g_i^X (x), f_i \circ g_i^X(x')) \\
  \displaybreak[0]
  &\leq \eps_i + (\alpha \cdot d_{X_i}(g_i^X (x), g_i^X(x'))) \\
  &\leq \eps_i + (\alpha \cdot (\eps_i + d_X(x,x'))) \\
  &= \alpha \cdot d_X(x,x') + (\alpha+1) \cdot \eps_i.
 \end{align*}

 As in the proof of \autoref{lem_GH:compact_subets_converge},
 the $h_i(x)$ do not have to converge. 
 Therefore, a diagonal argument on a dense subset of $K$ will be used 
 to construct a limit map which can be extended using the completeness of the limit space.

 Let $A = \{x_j \mid j \in \nn\}$ be a countable dense subset of $K$.
 Then $h_i(x_j) \in K'$ for all $i,j \in \nn$, and since $K'$ is compact,
 by a diagonal argument, there is a subsequence $(i_n)_{n \in \nn}$ 
 such that $(h_{i_n}(x_j))_{n \in \nn}$ converges for every $j \in \nn$. 
 Define $f: A \to K'$ by $f(x_j) = \lim_{n \to \infty} h_{i_n}(x_j)$. 
 This map is $\alpha$-bi-Lipschitz:
 For arbitrary $j,l \in \nn$, with the above estimate,
 \begin{align*}
  d_Y(f(x_j), f(x_l))
  &= \lim_{n \to \infty} d_Y(h_{i_n}(x_j), h_{i_n}(x_l))\\
  &\leq \lim_{n \to \infty} (\alpha+1) \cdot \eps_{i_n} + \alpha \cdot d_X(x_j,x_l) \\
  &= \alpha \cdot d_X(x_j,x_l).
 \end{align*}
 Analogously, 
 $d_Y(f(x_j), f(x_l)) \geq \frac{1}{\alpha} \cdot d_X(x_j,x_l).$

 Since $A$ is a countable dense subset of $K$, 
 $f$ can be extended to an $\alpha$-bi-Lipschitz map $f: K \to K'$
 (cf.~\autoref{lem:extending_Lipschitz_maps})
 where $f(x) = \lim_{l \to \infty} f(x_{j_l})$ 
 for $x \in K$ and $x_{j_l} \in A$ with $x_{j_l} \to x$.
 In particular, for $n \in \nn$ and $l \in \nn$,
 \begin{align*}
  &d_Y(f(x), h_{i_n}(x))\\
  &\leq d_Y(f(x), f(x_{j_l})) 
      + d_Y(f(x_{j_l}),h_{i_n}(x_{j_l})) 
      + d_Y(h_{i_n}(x_{j_l}), h_{i_n}(x)) \\
  &\leq d_Y(f(x), f(x_{j_l})) 
      + d_Y(f(x_{j_l}),h_{i_n}(x_{j_l})) 
      + \alpha \cdot d_X(x_{j_l}, x) 
      + (\alpha+1) \cdot \eps_{i_n} \\
  &\to d_Y(f(x), f(x_{j_l})) 
      + \alpha \cdot d_X(x_{j_l}, x) \as n \to \infty \\
  &\to 0 \as l \to \infty.
 \end{align*}
 Hence, $f(x) = \lim_{n \to \infty} h_{i_n}(x)$.

 Moreover, observe the following: Since $f_i$ is $\alpha$-bi-Lipschitz, it is injective.
 Therefore,
 the inverse $f_i^{-1}$ of $f_i$ exists on $f_i(K_i) \supseteq \im(g_i^X)$ 
 and is $\alpha$-bi-Lipschitz as well.
 Hence, for $x \in K$ and $y \in K'$,
 \begin{align*}
  d_Y(h_i(x),y)
  &= d_Y(f_i^Y \circ f_i \circ g_i^X(x),y) \\
  \displaybreak[0]
  &\leq 2 \eps_i + d_{Y_i}(f_i \circ g_i^X(x),g_i^Y(y)) \\
  \displaybreak[0]
  &\leq 2 \eps_i + \alpha \cdot d_{X_i}(g_i^X(x),f_i^{-1} \circ g_i^Y(y)) \\
  &\leq 2 \eps_i + \alpha \cdot (2 \eps_i + d_{X_i}(x,f_i^X \circ f_i^{-1} \circ g_i^Y(y))) \\
  &= 2 (\alpha +1) \eps_i + \alpha \cdot d_{X_i}(x,h_i'(y))
 \end{align*}
 where $h'_i := f_i^X \circ f_i^{-1} \circ g_i^Y$.
 With analogous arguments 
 and using a further subsequence $(i_{n_m})_{m \in \nn}$ of $(i_n)_{n \in \nn}$,
 there is an $\alpha$-bi-Lipschitz map $g : K' \to K$ 
 with $g(y) = \lim_{m \to \infty} h_{i_{n_m}}'(y)$ for all $y \in K'$.
 In particular, for all $y \in K'$,
 \begin{align*}
  d_Y(f \circ g(y),y)
  &= \lim_{m \to \infty} d_Y(h_{i_{n_m}}(g(y)),y)\\
  &\leq \lim_{m \to \infty} 2(\alpha+1)\eps_{i_{n_m}} 
      + \alpha \cdot d_X(g(y),h_{i_{n_m}}'(y))\\
  &=0.
 \end{align*}
 Thus, $f \circ g = \id_{K'}$. 
 Hence, $K' \subseteq \im(f)$ which proves $K' = f(K)$.
 In fact, with analogous argumentation, one can prove $g \circ f = \id_K$,
 i.e.~$g$ is the inverse of $f$.
 This proves the first part.

 \par\smallskip\noindent\ref{lem:Lipschitz-maps_converge_to_Lipschitz-limit-map--b}
 The proof of the second statement is based on the first part 
 and is done with very similar methods.
  
 Let $(f_i^X, g_i^X) \in \Isom{\eps_i}(K_i,K)$ 
 and $(f_i^Y, g_i^Y) \in \Isom{\eps_i}(f_i(K_i),K')$ be as before.  
 Then $L_i := \overline{g_i^X(L)} \subseteq K_i$ is a compact subset of $K_i$. 
 The proof of the subconvergences will be done in two steps:
 First, prove $L_i \to L$, then $f_i(L_i) \to f(L)$.
 For the maps defined below, 
 cf.~\autoref{pic:Lipschitz-maps_converge_to_Lipschitz-limit-map-b}.
 
 First, define $(\tilde{f}_i^X, \tilde{g}_i^X) \in \Isom{2\eps_i}(L_i,L)$ as follows:
 For $x_i \in g_i^X(L)$, choose a point $y \in L$ with $x_i = g_i^X(y)$; 
 for $x_i \in L_i \setminus g_i^X(L)$, 
 choose $y \in L$ with $d_{X_i}(x_i,g_i^X(y)) < \frac{\eps_i}{2}$. 
 Then define $\tilde{f}_i^X(x_i) := y$. 
 Finally, set $\tilde{g}_i^X := g_i^X$.
 By definition, 
 \[
  d_{X_i}(\tilde{g}_i^X \circ \tilde{f}_i^X(x_i),x_i) 
  = d_{X_i}(g_i^X \circ \tilde{f}_i^X(x_i),x_i) < \frac{\eps_i}{2}
 \] 
 for all $x_i \in L_i$.
 Conversely, for $x \in L$ and by applying this inequality,
 \begin{align*}
  d_X(\tilde{f}_i^X \circ \tilde{g}_i^X(x),x)
  \displaybreak[0]
  &= d_X(\tilde{f}_i^X \circ g_i^X(x),x) \\
  &\leq d_{X_i}(g_i^X \circ \tilde{f}_i^X (g_i^X(x))),g_i^X(x)) + \eps_i
  \\&\leq \frac{3}{2} \eps_i.
 \end{align*}
 Now let $x_i,x_i' \in L_i$ be arbitrary. Then 
 \begin{align*}
  &|d_X(\tilde{f}_i^X(x_i),\tilde{f}_i^X(x_i')) - d_{X_i}(x_i,x_i')| \\
  &\leq |d_X(\tilde{f}_i^X(x_i),\tilde{f}_i^X(x_i')) 
         - d_{X_i}(g_i^X(\tilde{f}_i^X(x_i)),g_i^X(\tilde{f}_i^X(x_i')))| 
  \\&\quad + |d_{X_i}(g_i^X(\tilde{f}_i^X(x_i)),g_i^X(\tilde{f}_i^X(x_i'))) 
         - d_{X_i}(x_i,x_i')|\\
  &< \eps_i + d_{X_i}(g_i^X \circ \tilde{f}_i^X(x_i),x_i) 
         + d_{X_i}(g_i^X \circ \tilde{f}_i^X(x_i'),x_i') \\
  &< 2 \eps_i.
 \end{align*}
 For $x,x' \in L$, 
 by definition,
 \[
  |d_{X_i}(\tilde{g}_i^X(x),\tilde{g}_i^X(x')) - d_X(x,x')| 
  < \eps_i 
  < 2 \eps_i,
 \] 
 and this proves 
 $(\tilde{f}_i^X, \tilde{g}_i^X) \in \Isom{2\eps_i}(L_i,L)$.
 \begin{figure}[t]
  \figuretwo
 \end{figure}

 In order to prove the subconvergence of $f_i(L_i)$ to $f(L)$, 
 observe that the compactness of $L_i$ and $L$, respectively, 
 and the continuity of $f_i$ and $f$, respectively, 
 prove the compactness of $f_i(L_i)$ and $f(L)$, respectively.
 
 Let 
 \[\delta_i(x) := d_Y(h_i(x),f(x))\] for $x \in L$ and 
 \[\delta_i := \sup_{x \in L} \delta_i(x).\]
 For the subsequence $(i_n)_{n \in \nn}$ of the first part, $\delta_{i_n}(x)$ converges to $0$.
 Then $\delta_{i_n}$ converges to $0$ as well: 
 Assume this is not the case,
 i.e.~there is $\epsilon > 0$ such that for all $l \in \nn$ 
 there exists $i_{n_l} \in \nn$ and $x_{n_l} \in X$ 
 with $\delta_{i_{n_l}}(x_{n_l}) \geq \eps$.
 After passing to a subsequence, 
 there is $x \in X$ such that $x_{n_l} \to x$ as $l \to \infty$. 
 Then
 \begin{align*}
  \eps
  &\leq \delta_{i_{n_l}}(x_{n_l}) \\
  &= d_Y(h_{i_{n_l}}(x_{n_l}),f(x_{n_l})) \\
  &\leq d_Y(h_{i_{n_l}}(x_{n_l}),h_{i_{n_l}}(x)) 
    + d_Y(h_{i_{n_l}}(x),f(x)) 
    + d_Y(f(x),f(x_{n_l})) \\
  &\leq (\alpha \cdot d_X(x_{n_l},x) 
    + (\alpha+1) \cdot \eps_{i_{n_l}}) 
    + \delta_{i_{n_l}}(x) 
    + \alpha \cdot d_X(x,x_{n_l}) \\
  &\to 0 \as l \to \infty.
 \end{align*}
 This is a contradiction. 

 Construct $(\tilde{f}_i^Y, \tilde{g}_i^Y) \in \Isom{\tilde{\eps_i}}(f_i(L_i),f(L))$ 
 for $\tilde{\eps}_i := (4\alpha+1)\eps_i + 2 \delta_i$ as follows:  
 Define $\tilde{f}_i^Y := f \circ \tilde{f}_i^X \circ f_i^{-1}$ 
 and $\tilde{g}_i^Y := f_i \circ g_i^X \circ f^{-1}$ 
 (recall that $f_i^{-1}$ exists on $f_i(L_i) \subseteq f_i(K_i)$ 
 and that $f : K \to K'$ is bijective).
      
 First, let $y_i \in L_i$ and $y \in L$ be arbitrary. 
 Then
 \begin{align*}
  d_{Y_i}(\tilde{g}_i^Y \circ \tilde{f}_i^Y (y_i),y_i)
  &= d_{Y_i}(f_i \circ g_i^X \circ \tilde{f}_i^X \circ f_i^{-1} (y_i),y_i) \\
  &\leq \alpha \cdot d_{X_i}(g_i^X \circ \tilde{f}_i^X (f_i^{-1} (y_i)),f_i^{-1}(y_i)) \\
  &< \alpha \cdot 2\eps_i \leq \tilde{\eps}_i ,  
 \end{align*}
 and completely analogously,
 \begin{align*}
  d_Y(\tilde{f}_i^Y \circ \tilde{g}_i^Y (y),y)
  &= d_Y(f \circ \tilde{f}_i^X \circ g_i^X \circ f^{-1} (y),y)
  <2 \alpha \eps_i\leq \tilde{\eps}_i.
 \end{align*}
 For $y,y' \in L$,
 \begin{align*}
  &|d_{Y_i}(\tilde{g}_i^Y(y),\tilde{g}_i^Y(y')) - d_Y(y,y')|\\
  &\leq |d_{Y_i}(\tilde{g}_i^Y(y),\tilde{g}_i^Y(y')) 
     - d_Y(f_i^Y \circ \tilde{g}_i^Y(y),f_i^Y \circ \tilde{g}_i^Y(y'))|
  \\&\quad 
     + |d_Y(f_i^Y \circ f_i \circ g_i^X \circ f^{-1}(y),
        f_i^Y \circ f_i \circ g_i^X \circ f^{-1}(y')) 
     - d_Y(y,y')|\\
  &< \eps_i 
     + d_Y(h_i \circ f^{-1}(y),f \circ f^{-1}(y)) 
     + d_Y(h_i \circ f^{-1}(y'),f \circ f^{-1}(y'))\\
  &\leq \eps_i + 2 \delta_i \leq \tilde{\eps}_i.
 \end{align*}
 Finally, let $y_i,y_i' \in Y_i$. Using the above estimates,
 \begin{align*}
  &|d_Y(\tilde{f}_i^Y(y_i),\tilde{f}_i^Y(y_i')) - d_{Y_i}(y_i,y_i')|\\
  &\leq |d_Y(\tilde{f}_i^Y(y_i),\tilde{f}_i^Y(y_i')) 
     - d_{Y_i}(\tilde{g}_i^Y(\tilde{f}_i^Y(y_i)),\tilde{g}_i^Y(\tilde{f}_i^Y(y_i')))|
  \\&\quad + |d_{Y_i}(\tilde{g}_i^Y(\tilde{f}_i^Y(y_i)),\tilde{g}_i^Y(\tilde{f}_i^Y(y_i'))) 
     - d_{Y_i}(y_i,y_i')|
  \displaybreak[0]\\
  &< \eps_i + 2 \delta_i
    + d_{Y_i}(\tilde{g}_i^Y(\tilde{f}_i^Y(y_i)),y_i) 
    + d_{Y_i}(\tilde{g}_i^Y(\tilde{f}_i^Y(y_i')),y_i')\\
  &\leq \eps_i + 2 \delta_i + 2 \cdot 2\alpha\eps_i 
  = \tilde{\eps}_i.
 \end{align*}
 Thus, $(\tilde{f}_i^Y, \tilde{g}_i^Y) \in \Isom{\tilde{\eps}_i}(f_i(L_i),f(L))$.
 Since $\tilde{\eps}_{i_n} \to 0 \as {n \to\infty}$, 
 this proves $f_{i_n}(L_{i_n}) \to f(L) \as {n \to\infty}$.
\end{proof}

\begin{lemma}\label{lem:extending_Lipschitz_maps}
 Let $(X,d_X)$ and $(Y,d_Y)$ be metric spaces where $Y$ is complete, 
 let $A \subseteq X$ 
 and $f : A \to Y$ be $\alpha$-(bi)-Lipschitz for some $\alpha > 0$.
 Then $f$ can be extended 
 to an $\alpha$-(bi)-Lipschitz map $\hat{f} : \bar{A} \to Y$.
\end{lemma}

\begin{proof}
 Let $a \in \bar{A}\setminus A$ be arbitrary. 
 Then there exists a (Cauchy) sequence $(a_n)_{n \in \nn}$ in $A$ converging to $a$. 
 By Lipschitz continuity of $f$, 
 $(f(a_n))_{n \in \nn}$ is a Cauchy sequence, 
 and thus has a limit $\hat{a}$ in the complete metric space $Y$.
 For any sequence $(\tilde{a}_n)_{n \in \nn}$ converging to $a$,
 $d_Y(f(a_n), f(\tilde{a}_n)) \leq \alpha \cdot d_X(a_n,\tilde{a}_n) \to 0$,
 i.e.~the limit $\hat{a}$ is independent of the choice of $(a_n)_{n \in \nn}$.
 Now define $\hat{f}(a) := \hat{a}$ for $a \in \bar{A}\setminus A$ 
 and $\hat{f}(a) := f(a)$ for $a \in A$. 
 For arbitrary $a,b \in A$ and sequences $a_n \to a$, $b_n \to b$ in $A$,
 \begin{align*}
  d_Y(\hat{f}(a),\hat{f}(b))
  &= \lim_{n \to \infty} d_Y(f(a_n),f(b_n))
  \\&\leq \lim_{n \to \infty} \alpha \cdot d_X(a_n,b_n)
  \\&= \alpha \cdot d(a,b).
 \end{align*}
 Hence, $\hat{f}$ is $\alpha$-Lipschitz.
 Analogously, if $f$ is $\alpha$-bi-Lipschitz, $\hat{f}$ is $\alpha$-bi-Lipschitz.
\end{proof}


\section{Ultralimits}\label{sec:ultralimits} 


Since sequences of proper spaces do not necessarily converge in the pointed \GH sense, 
a tool to enforce convergence can be useful. Such a tool are the so called ultralimits
since they always exist and are sublimits in the pointed \GH sense.
A basic reference from which the following definitions are taken 
is \cite[section I.5]{bridson-haefliger}.
Another, more set theoretical, reference is \cite[chapter 7]{jech}.
In the following, ultralimits will be introduced and some properties will be investigated.

\begin{defn}[{\cite[Definition I.5.47]{bridson-haefliger}}]
A \emph{non-principal ultrafilter on $\nn$} 
is a finitely additive probability measure $\omega$ on $\nn$ such that 
all subsets $S \subseteq \nn$ are $\omega$-measurable 
with $\omega(S) \in \{0,1\}$ and $\omega(S) = 0$ if $S$ is finite.
\end{defn}

\begin{rmk}
 If two sets have $\omega$-measure $1$, their intersection has $\omega$-measure $1$ as well:
 Let $\omega(A) = \omega(B) = 1$. 
 Then 
 \[
  \omega(\nn \setminus (A \cap B)) 
  = \omega(\nn \setminus A \cup \nn \setminus B) 
  \leq \omega(\nn \setminus A) + \omega (\nn \setminus B) 
  = 0,
 \]
 hence, $\omega(A \cap B) = 1$.
\end{rmk}

Using Zorn's Lemma, the existence of such a non-principal ultrafilter can be proven.
But even more is true: Given any infinite set, 
there exists a non-principal ultrafilter such that the set has measure $1$ 
with respect to this ultrafilter.

\begin{lemma}\label{lem_UL:existence_of_ultrafilter}
 Let $A \subseteq \nn$ be an infinite set. 
 Then there exists a non-principal ultrafilter $\omega$ on $\nn$ such that $\omega(A) = 1$.
\end{lemma}

\begin{proof}
 Let 
 \[
  G := \{B \subseteq \nn 
         \mid B \supseteq A~\text{or}~\nn \setminus B \text{ is finite}\}.
 \]
 For any $B_1, B_2 \in G$, the intersection $B_1 \cap B_2$ is non-empty: 
 This is obviously correct if both $B_j \supseteq A$ or both $\nn \setminus B_j$ are finite. 
 Thus, let $B_1 \supseteq A$ and $\nn \setminus B_2$ be finite: 
 Then $A \setminus B_2$ is finite as well,
 hence, $B_1 \cap B_2 \supseteq A \cap B_2 = A \setminus (A \setminus B_2)$ is infinite 
 since $A$ is infinite.
 In particular, the intersection is non-empty.
 
 Using that $G$ contains all sets with finite complement, 
 it follows from 
 \cite[Lemma 7.2 (iii)]{jech}, \cite[Theorem 7.5]{jech} and the subsequent remark therein
 that there exists a non-principal ultrafilter $\omega$ 
 such that $\omega(X) = 1$ for all $X \in G$.
 In particular, $\omega(A) = 1$. 
\end{proof}

Given a bounded sequence of real numbers, 
a non-principal ultrafilter provides a kind of \myquote{limit}.
In fact, these \myquote{limits} are accumulation points 
and non-principal ultrafilters pick out convergent subsequences.
\begin{lemma}[{\cite[Lemma I.5.49]{bridson-haefliger}}]
 Let $\omega$ be a non-principal ultrafilter on $\nn$. 
 For every bounded sequence of real numbers $(a_i)_{i \in \nn}$ 
 there exists a unique real number $l \in \rr$ 
 such that 
 \[\omega( \{i \in \nn \mid |a_i - l| < \eps \}) = 1\] 
 for every $\eps > 0$. 
 Denote this $l$ by $\limomega a_i$.
\end{lemma}

\begin{lemma}\label{lem_UL:ultrafilter_pick_cvgt_subsequence}
If $\omega$ is a non-principal ultrafilter on $\nn$ 
and $(a_i)_{i \in \nn}$ a bounded sequence of real numbers,
then $\limomega a_i$ is an accumulation point of $(a_i)_{i \in \nn}$.
Moreover, there exists a subsequence $(a_{i_j})_{j \in \nn}$ converging to $\limomega a_i$ 
such that $\omega(\{i_j \mid j \in \nn\}) = 1$.

Conversely, if $(a_i)_{i \in \nn}$ is a bounded sequence of real numbers 
and $a \in \rr$ any accumulation point,
then there exists a non-principal ultrafilter $\omega$ on $\nn$ such that $a = \limomega a_i$.
\end{lemma}

\begin{proof}
 Let $(a_i)_{i \in \nn}$ be any bounded sequence of real numbers.

 First, fix a non-principal ultrafilter $\omega$,
 let $a := \limomega a_i$ and 
 \[A_{\eps} := \{i \in \nn \mid |a_i - a| < \eps\}\] 
 for $\eps > 0$.
 By definition, $\omega(A_{\eps}) = 1$; 
 in particular, $A_{\eps}$ has infinitely many elements.
 Thus, $a$ is an accumulation point. 
 
 Next, prove that there exists $I \subseteq \nn$ with $\omega(I) = 1$
 such that the subsequence $(a_i)_{i \in I}$ converges to $a$.
 Assume this is not the case,
 i.e.~every $I \subseteq \nn$ satisfies $\omega(I) = 0$ 
 or $(a_i)_{i \in I}$ does not converge to $a$.
 Since $\omega(\nn) = 1$, $(a_i)_{i \in \nn}$ does not converge to $a$.
 Hence, there exists $\eps > 0$ 
 such that $A_{\eps}$ is finite.
 In particular, $\omega(A_{\eps}) = 0$ and this is a contradiction.
  
 Now let $J \subseteq \nn$ be a set of indices such that 
 $\omega(J) = 1$ and the subsequence $(a_j)_{j \in J}$ converges to $a$.
 By \autoref{lem_UL:existence_of_ultrafilter}, 
 there exists a non-principal ultrafilter $\omega$ such that $\omega(J) = 1$.
 By the first part, there exists a subsequence of indices $I \subseteq \nn$ 
 with $\omega(I) = 1$ and $a_j \to \limomega a_i$ as $j \to \infty$ for $j \in I$. 
 Now $\omega(I \cap J) = 1$ 
 and both $a_j \to a$ and $a_j \to \limomega a_i$ as $j \to \infty$ for $j \in I \cap J$.
 This proves $a = \limomega a_i$.
\end{proof}

An immediate consequence of the above lemma is the following: 
Given two bounded sequences of real numbers,
investigating sublimits coming from a common subsequence
and
investigating the \myquote{limits} with respect to the same non-principal ultrafilter
is the same.
\begin{lemma}\label{lem_UL:common-subseq=using-ultrafilter}
 Let $(a_i)_{i \in \nn}$ and $(b_i)_{i \in \nn}$ be bounded sequences of real numbers.
 \begin{enumerate}
 \item\label{lem_UL:common-subseq=using-ultrafilter--a} 
  If $\omega$ is a non-principal ultrafilter on $\nn$, 
  then there exists a subsequence $(i_j)_{j \in \nn}$ such that both
  $a_{i_j} \to \limomega a_i$ and $b_{i_j} \to \limomega b_i$ as $j \to \infty$.
 \item\label{lem_UL:common-subseq=using-ultrafilter--b} 
  If there are $a, b \in \rr$ and a subsequence $(i_j)_{j \in \nn}$ 
  such that both $a_{i_j} \to a$ and $b_{i_j} \to b$ as $j \to \infty$,
  then there exists a non-principal ultrafilter $\omega$ on $\nn$ 
  such that $a = \limomega a_i$ and $b = \limomega b_i$.
 \end{enumerate}
\end{lemma}

\begin{proof}
 \par\smallskip\noindent\ref{lem_UL:common-subseq=using-ultrafilter--a}
 By \autoref{lem_UL:ultrafilter_pick_cvgt_subsequence}, 
 there are subsequences of indices $I, J \subseteq \nn$ 
 with measures $\omega(I) = \omega(J) = 1$,
 \begin{align*}
  &a_j \to \limomega a_i \as j \to \infty~\text{for}~j \in I 
  \quad\aand\\
  &b_j \to \limomega b_i \as j \to \infty~\text{for}~j \in J.
 \end{align*}
 In particular, $I \cap J$ has $\omega$-measure $1$. 
 Hence, it is infinite and provides a common subsequence which satisfies the claim.
 
 \par\smallskip\noindent\ref{lem_UL:common-subseq=using-ultrafilter--b}
 This follows directly from the second part 
 of \autoref{lem_UL:ultrafilter_pick_cvgt_subsequence} 
 since the non-principal ultrafilter constructed there 
 depends only on the indices of the convergent subsequence.
\end{proof}

\begin{cor}
 Let $(a_i)_{i \in \nn}$ and $(b_i)_{i \in \nn}$ be bounded sequences of real numbers.
 \begin{enumerate}
  \item If $a_i \leq b_i$ for all $i \in \nn$, then $\limomega a_i \leq \limomega b_i$.
  \item $\limomega (a_i + b_i) = \limomega a_i + \limomega b_i$.
 \end{enumerate}
\end{cor}

\begin{proof}
 Observe that \autoref{lem_UL:common-subseq=using-ultrafilter} holds not only for two 
 but finitely many sequences for real numbers.
 Applying this and the corresponding statements for limits of sequences of real numbers 
 implies the claim.
\end{proof}


An ultralimit is a \myquote{limit space} assigned to a (pointed) sequence of metric spaces 
by using a non-principal ultrafilter.
The construction of this ultralimit is related to \GH convergence in the sense that
such a limit space is a sublimit in the pointed \GH sense.
On the other hand, given any sublimit in the pointed \GH sense,
there exists a non-principal ultrafilter 
such that the corresponding ultralimit is exactly this sublimit. 
This fact can be extended to a similar statement about finitely many different sequences 
and corresponding sublimits coming from a common subsequence.

\begin{defn}[{\cite[Definition I.5.50]{bridson-haefliger}}]
 Let $\omega$ be a non-principal ultrafilter on $\nn$,
 $(X_i,d_i,p_i)$, $i \in \nn$, be pointed metric spaces and 
 \[
  X_\omega := \{ [(x_i)_{i \in \nn}] 
  \mid x_i \in X_i~\aand~\sup\nolimits_{i \in \nn} d_i(x_i,p_i) < \infty\}
 \]
 where 
 \[(x_i)_{i \in \nn} \sim (y_i)_{i \in \nn}~\text{if and only if}~\limomega d_i(x_i,y_i) = 0.\]
 Furthermore, 
 let $d_\omega([(x_i)_{i \in \nn}],[(y_i)_{i \in \nn}]) := \limomega d_i(x_i,y_i)$. 
 Then $(X_\omega, d_\omega)$ is a metric space, 
 called \emph{ultralimit} of $(X_i,d_i,p_i)$ and denoted by $\limomega (X_i,d_i,p_i)$. 
\end{defn}

\begin{rmk}
 Let $\omega$ be a non-principal ultrafilter on $\nn$, 
 $(X_i,d_i,p_i)$, $i \in \nn$, be pointed metric spaces and $Y_i \subseteq X_i$.
 The limit $(Y_\omega, d_{Y_\omega}) := \limomega (Y_i,d_i,p_i)$ 
 is canonically a subset of $(X_\omega, d_{X_\omega}) := \limomega (X_i,d_i,p_i)$: 
 Obviously, 
 \begin{align*}
 &\{(y_i)_{i \in \nn} 
     \mid y_i \in Y_i~\aand~\sup\nolimits_i d_i(y_i,p_i) < \infty\} \\
 &\subseteq \{ (x_i)_{i \in \nn} 
     \mid x_i \in X_i~\aand~\sup\nolimits_i d_i(x_i,p_i) < \infty\}.
 \end{align*}
 Since the metric is the same on both $X_i$ and $Y_i$ 
 and since the equivalence classes are only defined by using the ultrafilter and the metric,
 $Y_\omega \subseteq X_\omega$.
 With the same argumentation, the metric coincides: For $y_i,y_i' \in Y_i$,
 \begin{align*}
  &d_{Y_\omega}([(y_i)_{i \in \nn}]_{Y_\omega},[(y_i')_{i \in \nn}]_{Y_\omega}) 
  \\&= \limomega d_i(y_i,y_i') 
  \\&= d_{X_\omega}([(y_i)_{i \in \nn}]_{X_\omega},[(y_i')_{i \in \nn}]_{X_\omega}).
 \end{align*}
\end{rmk}

\begin{lemma}[{\cite[Lemma I.5.53]{bridson-haefliger}}]
 The ultralimit of a sequence of metric spaces is complete.
\end{lemma}

In order to prove the correspondence of sublimits and ultralimits, 
first, compact metric spaces are investigated.

\begin{prop}\label{prop_UL:ultralimits_are_sublimits_cpt}
 Let $\omega$ be a non-principal ultrafilter on $\nn$ and
 $(X_i,d_i,p_i)$, $i \in \nn$, be pointed compact metric spaces 
 with compact ultralimit $(X_\omega,d_\omega)$ 
 and define $p_\omega := [(p_i)_{i \in \nn}] \in X_\omega$.
 Then $\limomega \dGH((X_i,p_i),(X_\omega,p_\omega)) = 0$.
\end{prop}

\begin{proof}
 The statement will be proven by using $\eps$-nets:
 First, finite $\eps$-nets in $X_i$ will be fixed 
 and it will be proven that their ultralimit is a finite $\eps$-net in $X_\omega$.
 Then the \GH distance of these nets will be estimated.
 Finally, using the triangle inequality and $\eps \to 0$ prove the claim.

 Fix $\eps > 0$.
 For every $i \in \nn$, 
 fix a finite $\eps$-net $A_i^{\eps} = \{a_i^1,\dots,a_i^{n_i}\}$ in the compact space $X_i$ 
 with $a_i^1 = p_i$, 
 i.e.~$d(a_i^k,a_i^l) \geq \eps$ for all $k \ne l$ and $X_i = \bigcup_{j=1}^{n_i} B_{\eps}(a_i^j)$.
 Let $A_\omega^{\eps}$ be the ultralimit of these $A_i^{\eps}$,
 i.e.
 \[
  A_\omega^{\eps} 
  = \{[(a_i)_{i \in \nn}] \mid \forall i \in \nn \, \exists 1 \leq j_i \leq n_i: a_i = a_i^{j_i}\} 
  \subseteq X_\omega,
 \]
 and let $p_\omega := [(p_i)_{i \in \nn}] \in A_\omega^{\eps}$.
 Then $A_\omega^{\eps}$ is again a finite $\eps$-net in $X_\omega$: 

 Let $[(a_i^{k_i})_{i \in \nn}], [(a_i^{l_i})_{i \in \nn}] \in A_\omega^{\eps}$.
 By definition, 
 \[[(a_i^{k_i})_{i \in \nn}] = [(a_i^{l_i})_{i \in \nn}]
 \text{ if and only if }
 \limomega d_i(a_i^{k_i},a_i^{l_i}) = 0.\]
 Since $d_i(a_i^{k_i},a_i^{l_i}) = 0$ exactly for those $i$ with $k_i = l_i$ 
 and $d_i(a_i^{k_i},a_i^{l_i}) \geq \eps$ otherwise,
 this implies 
 \[[(a_i^{k_i})_{i \in \nn}] = [(a_i^{l_i})_{i \in \nn}]
 \text{ if and only if }
 \omega(\{i \in \nn \mid k_i = l_i \}) = 1.\]
 In particular, for $[(a_i^{k_i})_{i \in \nn}] \ne [(a_i^{l_i})_{i \in \nn}]$,
 \[d_{X_\omega}([(a_i^{k_i})_{i \in \nn}] , [(a_i^{l_i})_{i \in \nn}]) 
 =\limomega d_i(a_i^{k_i},a_i^{l_i}) 
 \geq \eps.\]
 Furthermore, for arbitrary $[(x_i)_{i \in \nn}]$ there are $a_i^{j_i}$ 
 such that $x_i \in B_{\eps}(a_i^{j_i})$.
 Thus,
 \begin{align*}
  d_\omega([(x_i)_{i \in \nn}], [(a_i^{j_i})_{i \in \nn}])
  = \limomega d_i(x_i, a_i^{j_i})
  < \eps.
 \end{align*}
 This proves that $A_\omega^{\eps}$ is an $\eps$-net in $X_\omega$. 
 It remains to prove that $A_\omega^{\eps}$ is finite:
 Assume it is not. 
 Then $\bigcup_{p \in A_\omega^{\eps}} B_{\eps}(p)$ is an open cover of $X_\omega$, 
 and thus, has a finite subcover $X_\omega = \bigcup_{j=1}^k B_{\eps}(q_j)$ 
 with $q_j \in A_\omega^{\eps}$.
 Hence, for any $q \in A_\omega^{\eps} \setminus \{q_1,\dots,q_k\}$ 
 there exists $q_j$ such that $q \in B_{\eps}(q_j)$. 
 This is a contradiction to $d_\omega(q,q_j) \ge \eps$.

 Let $n_\omega < \infty$ denote the cardinality of $A_\omega^{\eps}$
 and $I := \{i \in \nn \mid n_i = n_\omega\}$ be those indices 
 such that $A_i^{\eps}$ and $A_\omega^{\eps}$ have the same cardinality.
 Then $\omega(I) = 1$:
 
 Let $A_\omega^{\eps} = \{z_1,\dots,z_{n_\omega}\}$ 
 and $z_k = [(a_i^{j_i^k})_{i \in \nn}]$ 
 where $1 \leq j_i^k \leq n_i$ for each $1 \leq k \leq n_{\omega}$.
 For $k \ne l$,
 one has $1 = \omega(\{i \in \nn \mid j_i^k \ne j_i^l\})$.
 Thus, 
 \begin{align*}
 1 
  &= \omega\big(\bigcap\nolimits_{1 \leq k < l \leq n_\omega} 
      \{i \in \nn \mid j_i^k \ne j_i^l\}\big)\\
  &= \omega(\{i \in \nn \mid \forall 1 \leq k < l \leq n_\omega : j_i^k \ne j_i^l\}) \\
  &\geq \omega(\{i \in \nn \mid n_\omega \leq n_i\}) \\
  &= \omega (I \cup J)
 \end{align*}
 where $J := \{i \in \nn \mid n_i > n_\omega\}$. 
 Assume $\omega(J) = 1$.
 For all $1 \leq j \leq n_\omega + 1$, let 
 \[ q_i^j :=
  \begin{cases}
   a_i^j	& \text{if } i \in J, \\
   p_i		& \text{if } i \notin J \\
  \end{cases}
 \]
 and $\tilde{z}_j := [(q_i^j)_{i \in \nn}] \in A_\omega^{\eps}$.
 By definition,
 $q_i^j = q_i^l$ if and only if $k \ne l$ or $i \in I$.
 Hence, if $k \ne l$, then
 $\omega(\{i \in \nn \mid q_i^k = q_i^l\})
 = \omega(\nn \setminus J) 
 = 1 - \omega(J) 
 = 0$.
 Thus, $\tilde{z}_k \ne \tilde{z}_l$ 
 and $\{\tilde{z}_1, \dots, \tilde{z}_{n_\omega+1}\} \subseteq A_\omega^{\eps}$,
 hence, $n_\omega+1 \leq n_\omega$. This is a contradiction.
 Therefore, $\omega (J) = 0$ and $\omega(I) = \omega(I \cup J) = 1$.

 Similarly,
 for all $1 \leq j \leq n_\omega$, let 
 \[ p_i^j :=
  \begin{cases}
   a_i^j	& \text{if } i \in I, \\
   p_i		& \text{if } i \notin I \\
  \end{cases}
 \]
 and $y_j := [(p_i^j)_{i \in \nn}] \in A_\omega^{\eps}$.
 Analogously, $y_k = y_l$ if and only if $k = l$.
 This implies $A_\omega^{\eps} = \{ y_1, \dots, y_{n_\omega} \} $.
 In particular, $y_1 = p_\omega$.

 For $1 \leq k < l \leq n_\omega$, 
 define 
 \begin{align*}
  I^{kl}_\delta 
  := &\{ i \in I \mid |d_\omega(y_k,y_l) - d_i(a_i^k, a_i^l)| < \delta \} 
  \\= &\{ i \in I \mid |d_\omega(y_k,y_l) - d_i(p_i^k, p_i^l)| < \delta \}.
 \end{align*}
 Since $d_\omega(y_k,y_l) = \limomega d_i(p_i^k, p_i^l)$ by definition,
 $\omega(I^{kl}_\delta) = 1$ for any $\delta > 0$.
 Therefore, $\limomega \delta_i^{kl} = 0$ 
 for $\delta_i^{kl} := |d_\omega(y_k,y_l) - d_i(a_i^k, a_i^l)|$.
 Thus, $\limomega \eps_i = 0$
 where $\eps_i := \max\{ \delta_i^{kl} \mid 1 \leq k < l \leq n_\omega\}$ for $i \in I$
 and $\eps_i := 0$ for $i \notin I$. 

 Let $i \in I$ be fixed 
 and define $f_{i} : A_{i}^{\eps} \to A_\omega^{\eps}$ 
 and $g_{i} : A_\omega^{\eps} \to A_{i}^{\eps}$ by
 \[ f_{i}(a_i^j) := y_j~\aand~g_{i}(y_j) = a_i^j\]
 for $1 \leq j \leq n_{\omega}$.
 In particular, 
 \[
  f_{i}(p_i) = f_{i}(a_i^1) = y_1 = p_\omega
  \quad\aand\quad
  g_{i}(p_\omega) = g_{i}(y_1) = a_i^1 = p_i.
 \]
 Obviously, $f_{i} \circ g_{i} = \id_{A_\omega^{\eps}}$ and $g_{i} \circ f_{i} = \id_{A_{i}^{\eps}}$.
 Further, for $1 \leq k < l \leq n_{\omega}$,
 \begin{align*}
  &|d_\omega(f_{i}(a_{i}^k), f_{i}(a_{i}^l)) - d_{i}(a_{i}^k, a_{i}^l)|
  = |d_\omega(y_k, y_l) - d_{i}(a_{i}^k, a_{i}^l)|
  = \delta_{i}^{kl} 
  \leq \eps_{i},
  \intertext{and analogously,}
  &|d_{i}(g_{i}(y_k), g_{i}(y_l)) - d_\omega(y_k,y_l)| \leq \eps_{i},
 \end{align*}
 i.e.~$(f_{i},g_{i}) \in \Isom{\eps_{i}}((A_{i}^{\eps},p_i),(A_\omega^{\eps},p_\omega))$.
 Thus, $\dGH((A_{i}^{\eps}, p_{i}), (A_\omega^{\eps}, p_\omega)) \leq 2 \eps_i$ 
 for any $i \in I$.

 For any compact metric space $(Z,d_Z)$ and $\eps$-net $A \subseteq Z$,
 \begin{align*}
  \dH^{d_Z}(Z,A) 
  = \inf \{r > 0 \mid B_r(A) \supseteq Z = B_{\eps}(A)\}
  \leq \eps.
 \end{align*}
 Hence, for any $p \in A$, $\dGH((A,p),(Z,p)) \leq d_H(Z,A) + d_Z(p,p) \leq \eps$.

 Applying this general statement, for fixed $i \in I$ and $\eps > 0$,
 \begin{align*}
  &\dGH((X_{i},p_{i}), (X_\omega, p_\omega)) \\
  &\leq \dGH((X_{i},p_{i}), (A_{i}^{\eps}, p_{i})) 
   \\&\quad + \dGH((A_{i}^{\eps}, p_{i}), (A_\omega^{\eps}, p_\omega)) 
   \\&\quad + \dGH((A_\omega^{\eps}, p_\omega), (X_\omega, p_\omega)) \\
  &\leq 2 \eps + 2 \eps_{i}.
 \end{align*}
 In particular, $\limomega \dGH((X_{i},p_{i}), (X_\omega, p_\omega)) \leq 2 \eps$.
 Since this holds for all $\eps > 0$, 
 \[\limomega \dGH((X_{i},p_{i}), (X_\omega, p_\omega)) = 0.\qedhere\]
\end{proof}

\begin{cor}
 Let $\omega$ be a non-principal ultrafilter on $\nn$.
 If the ultralimit of compact metric spaces is compact, it is a sublimit in the pointed \GH sense
 which comes from a subsequence with index set whose $\omega$-measure is $1$.
\end{cor}

\begin{proof}
 Let $(X_i,d_i,p_i)$, $i \in \nn$, be pointed compact metric spaces, 
 $(X_\omega,d_\omega)$ their compact ultralimit and $p_\omega = [(p_i)_{i \in \nn}]$.
 By the previous proposition, \[\limomega \dGH((X_i,p_i), (X_\omega, p_\omega)) = 0,\] 
 and by \autoref{lem_UL:ultrafilter_pick_cvgt_subsequence}, 
 there exists a subsequence $(i_j)_{j \in \nn}$ of natural numbers satisfying 
 $\omega(\{i_j \mid j \in \nn\}) = 1$ such that 
 \[\dGH((X_{i_j},p_{i_j}), (X_\omega, p_\omega)) \to 0 \as {j \to \infty}.\qedhere\]
\end{proof}

This result now gives a corresponding result for non-compact spaces.
\begin{prop}\label{prop_UL:ultralimits_are_sublimits_ncpt}
 Let $\omega$ be a non-principal ultrafilter on $\nn$.
 The ultralimit of a sequence of pointed proper length spaces 
 is a sublimit in the pointed \GH sense
 (which comes from a subsequence with index set of $\omega$-measure $1$).
 
 Conversely, 
 the sublimit of a sequence of pointed proper length spaces in the pointed \GH sense
 is the ultralimit with respect to a non-prin\-ci\-pal ultrafilter.
\end{prop}

\begin{proof}
 Let $(X_i,d_i,p_i)$, $i \in \nn$, be pointed proper length spaces, 
 $(X_\omega, d_\omega)$ the corresponding ultralimit 
 and $p_\omega := [(p_i)_{i \in \nn}] \in X_\omega$.
 First it will be shown that an $r$-ball in the ultralimit is the ultralimit of $r$-balls.
 Then applying the corresponding statement for compact sets proves the claim. 

 For $r > 0$, 
 let $X_\omega^r \subseteq X_\omega$ denote the ultralimit of $(\B_r^{X_i}(p_i), d_i, p_i)$. 
 This is a closed subset of $X_\omega$:
 First, observe 
 \[X_\omega^r = \{[(q_i)_{i \in \nn}] \mid q_i \in X_i~\aand~d_i(q_i,p_i) \leq r\}.\]
 Let $(z_n)_{n \in \nn}$ be a sequence in $X_\omega^r$ which converges to a limit $z \in X_\omega$. 
 Denote $z_n = [(q_i^n)_{i \in \nn}]$ and $z = [(q_i)_{i \in \nn}]$ 
 where $q_i^n, q_i \in X_i$ with $d_i(q_i^n,p_i) \leq r$ for all $i,n \in \nn$ 
 and $\sup_{i \in \nn} d_i(q_i,p_i) < \infty$.
 Moreover, $d_\omega(z_n,z) = \limomega d_i(q_i^n,q_i) \to 0$ as $n \to \infty$.
 For all $n \in \nn$,
 $d_\omega(z_n,p_\omega) = \limomega d_i(q_i^n,p_i) \leq r$.
 Hence,
 \begin{align*}
  d_\omega(z,p_\omega)
  &\leq \lim_{n \to \infty} d_\omega(z,z_n) + d_\omega(z_n,p_\omega)
  \leq r
 \end{align*}
 and $z \in X_\omega^r$. This proves that $X_\omega^r$ is closed.
 
 In fact, $X_\omega^r = \B_r^{X_\omega}(p_\omega)$:
 First, let $[(q_i)_{i\in\nn}] \in X_\omega^r \subseteq X_\omega$ be arbitrary.
 Since 
 \[
  d_\omega([(q_i)_{i\in\nn}],[(p_i)_{i\in\nn}])
  = \limomega d_i(p_i,q_i) 
  \leq r,
 \] 
 $[(q_i)_{i\in\nn}] \in \B_r^{X_\omega}(p_\omega)$.
  
 Now let $[(q_i)_{i\in\nn}] \in B_r^{X_\omega}(p_\omega)$ 
 and $I := \{i \in \nn \mid d_i(p_i,q_i) < r\}$.
 Define 
 \[ \tilde{q}_i :=
 \begin{cases}
   q_i & \text{if } i \in I,\\
   p_i & \text{if } i \notin I.
  \end{cases}
 \]
 By definition, $[(\tilde{q}_i)_{i\in\nn}] \in X_\omega^r$.
 Furthermore, $[(q_i)_{i\in\nn}] = [(\tilde{q}_i)_{i\in\nn}] \in X_\omega^r$:  
 Since $[(q_i)_{i\in\nn}] \in B_r^{X_\omega}(p_\omega)$, $0 \leq l:= \limomega d_i(q_i,p_i) < r$. 
 For $\delta := r-l > 0$, 
 \begin{align*}
  1
  &= \omega(\{ i \in \nn \mid |d_i(q_i,p_i) - l| < \delta \}) \\
  &\leq \omega(\{ i \in \nn \mid d_i(q_i,p_i) < l + \delta = r \}) \\
  &= \omega(I).
 \end{align*}
 Thus, for arbitrary $\eps > 0$,
 \begin{align*}
  \omega(\{ i \in \nn \mid d_i(q_i,\tilde{q}_i) < \eps\})
  &\geq \omega(\{i \in \nn \mid q_i = \tilde{q}_i\}) \\
  &= \omega(I)
  = 1.
 \end{align*}
 Therefore, $\limomega d_i(q_i,\tilde{q}_i) = 0$ 
 and $[(q_i)_{i\in\nn}] = [(\tilde{q}_i)_{i\in\nn}] \in X_\omega^r$.
 Consequently, 
 $B_r^{X_\omega}(p_\omega) \subseteq X_\omega^r$. 
 Since $X_\omega^r$ is closed,
 this proves 
 $\B_r^{X_\omega}(p_\omega) \subseteq X_\omega^r$, and hence, equality,
 i.e.~$\B_r^{X_\omega}(p_\omega) = \limomega (\B_r^{X_i}(p_i), d_i, p_i)$.

 For any $r > 0$ 
 and $\eps_i^r := \dGH((\B_r^{X_i}(p_i),p_i),( \B_r^{X_\omega}(p_\omega),p_\omega))$,
 $\limomega \eps_i^r = 0$ by \autoref{prop_UL:ultralimits_are_sublimits_cpt}.
 By \autoref{prop_UL:eps^(r_n)_n<=1/r_n--ultrafilter},
 there exists $r_i>0$ with 
 \[
  \limomega \frac{1}{r_i} = 0 
  \quad\aand\quad 
  \omega\Big(\Big\{i \in \nn \mid \eps_i^{r_i} \leq \frac{1}{r_i}\Big\}\Big) = 1.
 \]

 By \autoref{lem_UL:ultrafilter_pick_cvgt_subsequence}, 
 there is $J = \{i_1 < i_2 < \dots \} \subseteq \nn$ 
 such that $\omega(J) = 1$ and $r_{i_j} \to \infty$.
 Let \[I := J \cap \Big\{i \in \nn \mid \eps_i^{r_i} \leq \frac{1}{r_i}\Big\}.\]
 Then $\omega(I) = 1$ and $I = \{i_{j_1} < i_{j_2} < \dots\} \subseteq J$.
 Thus, $r_{i_{j_l}} \to \infty$ 
 and 
 \[
  \dGH((\B_{r_{i_{j_l}}}^{X_{i_{j_l}}}(p_{i_{j_l}}),p_{i_{j_l}}),
       (\B_{r_{i_{j_l}}}^{X_\omega}(p_\omega),p_\omega)) 
  = \eps_{i_{j_l}}^{r_{i_{j_l}}} 
  \leq \frac{1}{r_{i_{j_l}}} 
  \to 0
 \]
 as $l \to \infty$.
 Now \autoref{cor:dgh_small_iff_eps_approx_noncompact} proves 
 $(X_{i_{j_l}},p_{i_{j_l}}) \to (X_\omega, p_\omega)$ in the pointed \GH sense
 where $\omega(\{i_{j_l} \mid l \in \nn \}) = 1$
 and this finishes the proof of the first part.

 The proof of the second statement
 can be done completely analogously
 to the one of \autoref{lem_UL:ultrafilter_pick_cvgt_subsequence}.
\end{proof}

\begin{lemma}\label{prop_UL:eps^(r_n)_n<=1/r_n--ultrafilter}
 Let $\omega$ be an ultrafilter on $\nn$ 
 and for every $r > 0$ let $(\eps_i^r)_{i \in \nn}$ be a sequence 
 such that $\limomega \eps_i^r = 0$.
 Then there exists a sequence $(r_i)_{i \in \nn}$ of positive real numbers 
 such that $\limomega \frac{1}{r_i} = 0$ 
 and $\omega(\{i \in \nn \mid \eps_i^{r_i} \leq \frac{1}{r_i} \}) = 1$.
\end{lemma}

\begin{proof}
 For $i \in \nn$, let $R_i := \{r > 0 \mid \eps_i^r \leq \frac{1}{r}\}$. 
 The idea of this proof, similar to the one of \autoref{prop:eps^(r_n)_n<=h(1/r_n)}, 
 is to find a sequence $r_i \in R_i$ with $r_i > i$ for a set of indices of $\omega$-measure $1$.
 Since the $R_i$ need to be non-empty, let $I := \{i \in \nn \mid R_i \ne \emptyset\}$.
 Due to  $\limomega \eps_i^1 = 0$,
 \begin{align*}
  \omega(I)
  = \omega\big(\big\{i \in \nn \mid \exists\,r > 0 : \eps_i^r \leq \frac{1}{r} \big\}\big)
  \geq  \omega(\{i \in \nn \mid \eps_i^1 \leq 1\})
  = 1,
 \end{align*}
 i.e.~$\omega(I) = 1$. 
 Let $J := \{i \in \nn \mid \neg \exists\,C > 0 : R_i \subseteq [0,C]\}$ be 
 the indices of the unbounded sets. 
 In particular, $J \subseteq I$.
 In the following, the cases of $\omega(J) = 0$ and $\omega(J) = 1$ will be distinguished. 
 
 In advance, observe that for sets of indices of $\omega$-measure $1$ 
 the corresponding $R_i$ cannot have a uniform upper bound:
 Let $A \subseteq \nn$ be any subset such that there exists $C > 0$ 
 with $\bigcup_{i \in A} R_i \subseteq [0,C]$ 
 and let $r > C$. 
 Then $i \in A$ implies $r \notin R_i$, i.e.~$\eps_i^r > \frac{1}{r}$. 
 Thus, $\omega(A) \leq \omega(\{i \in \nn \mid \eps_i^r > \frac{1}{r}\}) = 0$.
 
 First, let $\omega(J) = 1$.
 For $i \in J$, choose $r_i \in R_i \cap (i, \infty)$. For $i \in \nn \setminus J$, let $r_i := 1$.
 Then 
 \[\omega\Big(\Big\{i \in \nn \mid \eps_i^{r_i} \leq \frac{1}{r_i}\Big\}\Big) 
 \geq \omega(\{i \in \nn \mid r_i\in R_i\}) 
 \geq \omega(J) = 1.\]
 For arbitrary $\eps > 0$, choose $N \in \nn$ with $\frac{1}{N} \leq \eps$. 
 For $i \in J$ with $i \geq N$, 
 \[\frac{1}{r_i} < \frac{1}{i} \leq \frac{1}{N} \leq \eps\]
 and 
 \[\omega\Big(\Big\{i \in \nn \mid \frac{1}{r_i} \leq \eps \Big\}\Big) 
 \geq \omega(J \cap [N,\infty)) = 1.\]
 Thus, $\limomega \frac{1}{r_i} = 0$ and $r_i$ has the desired properties.
 
 Now let $\omega(J) = 0$. 
 For $i \in I \cap J^c$, 
 let $s_i := \sup R_i$ denote the least upper bound of $R_i$ 
 and choose $r_i \in [\frac{s_i}{2},s_i] \cap R_i$. 
 For $i \in I^c \cup J$, 
 let $s_i := r_i := 1$.
 Then 
 \[
  \omega\Big(\Big\{i \in \nn \mid \eps_i^{r_i} \leq \frac{1}{r_i}\Big\}\Big) 
  \geq \omega(\{i \in \nn \mid r_i\in R_i\}) 
  \geq \omega(I \cap J^c) 
  = 1.
 \]
 Let $\eps > 0$ and $K_{\eps} := \{i \in I \cap J^c \mid \frac{1}{s_i} > \eps\}$.
 Then 
 \[
  \bigcup_{i \in K_{\eps}} R_i 
  \subseteq \bigcup_{i \in K_{\eps}} [0,s_i] 
  \subseteq \Big[0,\frac{1}{\eps}\Big],
 \]
 and thus, by the above argumentation, $\omega(K_{\eps}) = 0$.
 Then, using $\omega(I \cap J^c) = 1$, 
 \begin{align*}
  \omega\big(\big\{i \in \nn \mid \frac{1}{s_i} \leq \eps \big\}\big)
  &= 1 - \omega\big(\big\{i \in \nn \mid \frac{1}{s_i} > \eps \big\}\big) \\
  &= 1 - \omega\big(\big\{i \in I \cap J^c \mid \frac{1}{s_i} > \eps \big\}\big) \\
  &= 1 - \omega(K_{\eps}) 
  = 1.
 \end{align*}
 Hence, $\limomega \frac{1}{s_i} = 0$ and $\frac{1}{r_i} \leq \frac{2}{s_i}$ proves the claim.
\end{proof}

As for bounded sequences of real numbers, 
investigating sublimits coming from the same subsequence is the same as investigating ultralimits.
\begin{lemma}
 Let $(X_i,d_{X_i},p_i)$ and $(Y_i,d_{Y_i},q_i)$, $i \in \nn$, be pointed proper length spaces.
 \begin{enumerate}
 \item 
  Let $\omega$ be a non-principal ultrafilter on $\nn$. 
  Then there exists a subsequence $(i_j)_{j \in \nn}$ such that both
  \begin{align*}
   &(X_{i_j},p_{i_j}) \to \limomega (X_i,d_{X_i},p_i) \quad\aand\\
   &(Y_{i_j},q_{i_j}) \to \limomega (Y_i,d_{Y_i},q_i)
  \end{align*}
  in the pointed \GH sense
  as $j \to \infty$.
 \item 
  Let $(X,d_X,p)$ and $(Y,d_Y,q)$ be pointed length spaces 
  and $(i_j)_{j \in \nn}$ be a subsequence such that both
  \begin{align*}
   &(X_{i_j},p_{i_j}) \to (X,p) \quad\aand\\
   &(Y_{i_j},q_{i_j}) \to (Y,q)
  \end{align*}
  in the pointed \GH sense
  as $j \to \infty$.
  Then there exists a non-principal ultrafilter $\omega$ on $\nn$ 
  such that there are isometries
  \begin{align*}
   &\limomega (X_i,d_{X_i},p_i) \cong (X,p) 
   \quad\aand\\
   &\limomega (Y_i,d_{Y_i},q_i) \cong (Y,q).
  \end{align*}
 \end{enumerate}
\end{lemma}

\begin{proof}
 Using \autoref{prop_UL:ultralimits_are_sublimits_ncpt},
 the proof can be done completely analogously 
 to the one of \autoref{lem_UL:common-subseq=using-ultrafilter}.
\end{proof}



 
\bibliographystyle{amsalpha}


\end{document}